\UseRawInputEncoding

\documentclass[11pt,reqno]{amsart}

\usepackage[american]{babel}
\usepackage[portrait,a4paper,margin=3cm]{geometry}

\usepackage{amsmath,amssymb,mathrsfs,mathtools}
\mathtoolsset{showonlyrefs,showmanualtags}
\usepackage{dsfont}

\usepackage{tikz}
\usetikzlibrary{decorations.pathreplacing, positioning}
\usepackage{subfigure}
\usepackage{caption}
\usepackage{color}

\usepackage[pagebackref]{hyperref}

\numberwithin{equation}{section}
\numberwithin{figure}{section}

\newtheorem{theorem}{Theorem}[section]
\newtheorem{proposition}[theorem]{Proposition}
\newtheorem{lemma}[theorem]{Lemma}
\newtheorem{corollary}[theorem]{Corollary}
\newtheorem{remark}[theorem]{Remark}
\newtheorem{definition}[theorem]{Definition}
\newtheorem{example}[theorem]{Example}

\title[Nonlinear spin-exchange dynamics]{Kac's Program and Relative Entropy Decay for Nonlinear Spin-Exchange Dynamics}

\author{Pietro Caputo}
\address{Department of Mathematics and Physics, Roma Tre University, Largo San Murialdo 1, 00146 Roma, Italy.}
\email{pietro.caputo@uniroma3.it}

\author{Mario Morellini}
\address{Department of Mathematics and Physics, Roma Tre University, Largo San Murialdo 1, 00146 Roma, Italy.}
\email{mario.morellini@uniroma3.it}

\definecolor{darkblue}{rgb}{0.1,0.1,0.7}
\definecolor{darkred}{rgb}{0.7,0.1,0.1}

\DeclareMathSymbol{\leqslant}{\mathalpha}{AMSa}{"36}
\DeclareMathSymbol{\geqslant}{\mathalpha}{AMSa}{"3E}
\DeclareMathSymbol{\eset}{\mathalpha}{AMSb}{"3F}
\renewcommand{\leq}{\le}
\renewcommand{\geq}{\ge} 

\newcommand{\ent}{{\rm Ent}}
\newcommand{\cov}{{\rm Cov}}
\newcommand{\var}{{\rm Var}}

\newcommand{\tc}{\,|\,}

\newcommand{\scalar}[2]{\langle #1 , #2\rangle}

\let\a=\alpha \let\b=\beta \let\g=\gamma \let\d=\delta
\let\e=\varepsilon  \let\l=\lambda
 \let\r=\rho \let\s=\sigma \let\t=\tau
\let\m=\mu      
\let\si=\sigma    
\let\O=\Omega
\let\G=\Gamma
\let\L=\Lambda

\newcommand{\cA}{\ensuremath{\mathcal A}} 
\newcommand{\cB}{\ensuremath{\mathcal B}} 
\newcommand{\cC}{\ensuremath{\mathcal C}} 
\newcommand{\cD}{\ensuremath{\mathcal D}} 
\newcommand{\cE}{\ensuremath{\mathcal E}} 
\newcommand{\cF}{\ensuremath{\mathcal F}}

\newcommand{\cI}{\ensuremath{\mathcal I}} 
 
\newcommand{\cK}{\ensuremath{\mathcal K}} 
\newcommand{\cL}{\ensuremath{\mathcal L}} 
\newcommand{\cM}{\ensuremath{\mathcal M}}

\newcommand{\cP}{\ensuremath{\mathcal P}} 
\newcommand{\cQ}{\ensuremath{\mathcal Q}} 
 
\newcommand{\cS}{\ensuremath{\mathcal S}} 
\newcommand{\cT}{\ensuremath{\mathcal T}}

\newcommand{\bbE}{{\ensuremath{\mathbb E}} }

\newcommand{\bbN}{{\ensuremath{\mathbb N}} } 
 
\newcommand{\bbP}{{\ensuremath{\mathbb P}} } 
 
\newcommand{\bbR}{{\ensuremath{\mathbb R}} } 
 
\newcommand{\bbT}{{\ensuremath{\mathbb T}} }

\begin{document}

\begin{abstract}
We introduce and analyze a nonlinear exchange dynamics for Ising spin systems with arbitrary interactions. The evolution is governed by a quadratic Boltzmann-type equation that conserves the mean magnetization. Collisions are encoded through a spin-exchange kernel chosen so that the dynamics converge to the Ising model with the prescribed interaction and mean magnetization profile determined by the initial state. We prove a general convergence theorem, valid for any interaction and any transport kernel. 
Moreover, we show that, for sufficiently weak interactions, the system relaxes exponentially fast to equilibrium in relative entropy, with optimal decay rate independent of the initial condition. The proof relies on establishing a strong version of the Kac program from kinetic theory. In particular, we show that the associated Kac particle system satisfies a modified logarithmic Sobolev inequality with constants uniform in the number of particles. This is achieved by adapting the method of stochastic localization to the present conservative setting.
\end{abstract}

\maketitle
\thispagestyle{empty}

\tableofcontents

\section{Introduction and main results}

We consider Boltzmann-type evolution equations of the form  
\begin{equation}
\label{non-linear rqs0}
    \frac{d p_t}{dt} \,=\, p_t \circ p_t \,-\, p_t,
\end{equation}
where \( p_t \) is a probability distribution on a configuration space \( \Omega \), and the \emph{collision product} \( p_t \circ p_t \) is defined through a probability kernel \( Q(\tau \tc \sigma, \sigma') \).  
The kernel takes as input two independent configurations \( \sigma, \sigma' \in \Omega \), each sampled from \( p_t \), and produces an output configuration \( \tau \in \Omega \), with law
\begin{equation}
\label{non-linear rqs00}
 p_t \circ p_t(\tau) = \int p_t(d\sigma) \int p_t(d\sigma') \, Q(\tau \tc \sigma, \sigma') .
\end{equation}

A classical example is Kac's caricature of a Maxwellian gas~\cite{Kac_pr}, which captures essential features of the Boltzmann equation in a simplified setting; see also \cite{Bertini_et_al} for a discrete version of Kac's model. 
Such nonlinear equations admit a natural probabilistic representation in terms of random binary trees~\cite{Wild,mckean}.  
A key structural property is the existence of \emph{conservation laws}, associated to quantities that remain invariant under the collision product.  
Each conservation law gives rise to a family of stationary measures, which determine the possible limiting distributions of the dynamics.

For instance, in Kac's model, where \( \Omega = \mathbb{R} \) is the set of one-dimensional velocities and collisions conserve the kinetic energy, for  any initial distribution \( p_0 \) on \( \mathbb{R} \) with mean zero and finite variance one has convergence to the centered Gaussian distribution with the same variance as $p_0$.  
Understanding the \emph{rate of convergence to equilibrium} has been a longstanding question, and even in this simple setting it presents significant challenges~\cite{KAC,Vill1,Vill2,entfound,Dolera}.


In this work we focus on a \emph{combinatorial counterpart} of the Boltzmann equation, where \( \Omega \) is the Boolean cube, and the analogue of the Maxwellian equilibrium is played by a high-dimensional Gibbs measure, such as the Ising model on a graph.  
At infinite temperature, the model reduces to a version of the recombination equations commonly studied in population genetics~\cite{Sinetal,Sinetal2,CP2024nonlinear}.  
When interactions are present, however, the system becomes substantially richer, and analyzing its convergence to stationarity raises new challenges.

We consider a family of models in which a collision between two spin configurations consists of exchanging spin values at a pair of sites, selected according to a prescribed transport kernel.  
The conservation laws correspond to the magnetization profile of the distribution along the irreducible components of the transport kernel.  
By tuning the transport kernel, one can realize any desired magnetization profile.  

Our first result is a general convergence theorem, valid for arbitrary interaction strength and for all choices of transport kernel.  
We then turn to the question of quantitative convergence rates.  
A discrete-time version of the model, in the special case where the conservation law is given by the single-site magnetization profile, was recently analyzed in~\cite{AC2024nonlinear}, providing convergence rates in total variation distance at high temperature.  
Here we consider arbitrary magnetization profiles and obtain finer quantitative bounds in terms of relative entropy rather than total variation distance.


Our main tool is an adaptation of the \emph{Kac program} to this combinatorial setting, in the spirit of what has been developed for Kac's original model of a Maxwellian gas~\cite{entfound,kacsprog}.  
In contrast to that classical setting, we show that for the present model one can fully implement Kac's program in a strong quantitative sense.  
Specifically, under a suitable high-temperature assumption, we establish a uniformly positive rate of exponential decay of relative entropy, independent of the initial condition.  
To our knowledge, this provides the first example of a high-dimensional system with genuine interactions where such a strong form of Kac's program can be implemented.


We now proceed to a detailed description of the models and our main results.

\subsection{Spin system}
Let \(\Omega = \{-1, +1\}^n\), for some positive integer $n$, denote the set of spin configurations. 
We consider Ising systems defined through the Gibbs measure
\begin{equation}\label{Ising Model def}
    \mu_{J,h}(\sigma) = \frac{1}{Z_{J,h}} \exp\left(  \textstyle{\frac{1}{2}\sum_{i,j \in [n]}} J_{i,j} \sigma_i \sigma_j +  \textstyle{\sum_{i\in [n]}} h_i \sigma_i \right),\qquad \sigma\in\Omega,
\end{equation}
where \(J = (J_{i,j})_{i,j \in [n]}\) is the interaction matrix, \(h = (h_i)_{i \in [n]}\) is the external magnetic field, and \(Z_{J,h}\) is the partition function ensuring normalization, so that $\mu_{J,h}$ is a probability over $\Omega$. Throughout, we use the notation $[n]=\{1,\dots,n\}$. 
We assume that \(J\) is a real symmetric matrix, and $h$ is a real vector. When \(h=0\) we will simply write \(\mu_J\).
When $J=a A_G$, where $a\in\bbR$ and $A_G$ is the adjacency matrix of an undirected graph $G$ with $n$ vertices, then $\mu_{J,h}$ represents the Ising model on the graph $G$ with external field $h$, and with interaction strength $a$. If $a>0$ the model is called \emph{ferromagnetic}, while it is called \emph{antiferromagnetic} if $a<0$. 
A classical example is the Curie-Weiss model (or mean-field Ising model), where the interaction matrix is given by \(J_{i,j} = \beta/n\) for all \(i,j \in [n]\), with \(\beta > 0\) denoting the inverse temperature, which corresponds to the Ising model on the complete graph with interaction strength $a=\beta/n$. Other examples include disordered systems such as the Sherrington-Kirkpatrick model, where $J_{i,j}=\frac{\beta}{\sqrt n} g_{ij}$, where $g_{ij}=g_{ji}$, $i<j$, are i.i.d.\ standard normal random variables, or the Ising model on a random $d$-regular graph with interaction strength $a\in\bbR$, where $d\ge 3$ is a fixed integer, in which case $J=a A_G$, where $A_G$ is the adjacency matrix of a random  $d$-regular graph with $n$ vertices.


\subsection{The collision kernel}
A collision between two spin configurations $(\s,\s')\in\Omega\times\Omega$ consists in the exchange of the $\ell$-th spin of $\si$ with the $k$-th spin of $\si'$, that is 
\begin{equation}\label{gxk1}
(\sigma, \sigma') \mapsto \left(S_\ell(\sigma, \sigma'_k),\ S_k(\sigma', \sigma_\ell)\right),
\end{equation}
where \(S_\ell: \O\times\{-1,1\}\to \Omega \) is the configuration defined as
\[
S_\ell(\sigma, a)_i = 
\begin{cases}
    \sigma_i & \text{if } i \neq \ell, \\
    a & \text{if } i = \ell.
\end{cases}
\]
The sites $(\ell,k)$ at which the exchange occurs are chosen by picking $\ell$ uniformly at random in $[n]$ and then by picking $k$ according to a prescribed {\em symmetric stochastic $n\times n$ matrix},  denoted $\cK(\ell,k)$, which we refer to as the transport kernel.  In order to guarantee the reversibility of our system with respect to the Ising measure, the proposed exchange \eqref{gxk1} is accepted with probability 
 \begin{equation}\label{gxk3}
    p_J(\ell,k\,|\,\sigma,\sigma')=
       \frac{ \mu_J(S_\ell(\sigma, \sigma'_k))\, \mu_J(S_k(\sigma', \sigma_\ell))}{ \mu_J(\sigma)\, \mu_J(\sigma')+ \mu_J(S_{\ell}(\sigma, \sigma'_{k}))
      \, \mu_J(S_{k}(\sigma', \sigma_{\ell}))},
\end{equation}
and is otherwise rejected. 
Summarizing, the pair $(\s,\s')\in\Omega\times\Omega$ produces the new pair $(\t,\t')\in\Omega\times\Omega$ according to the probability kernel on $\Omega\times\Omega$ given by 
\begin{equation}\label{gxk2}
    \mathcal{Q}_J^\mathcal{K}(\sigma, \sigma'; \tau, \tau') = \frac1n\sum_{\ell, k = 1}^n \mathcal{K}(\ell, k)\ Q_J^{\ell,k}(\sigma, \sigma'; \tau, \tau'),
\end{equation}
where, for each pair of sites $\ell,k\in[n]$, $Q_J^{\ell,k}$ is the probability kernel on $\Omega\times\Omega$ defined as
\begin{align}\label{gxk4}
Q_J^{\ell,k}(\sigma, \sigma'; \tau, \tau')
     & =    p_J(\ell,k\,|\,\sigma,\sigma') \mathbf{1}(\tau = S_\ell(\sigma, \sigma'_k),\tau' = S_k(\sigma', \sigma_\ell)) \;+\;
    \nonumber \\ & \qquad\quad + \;( 1-  p_J(\ell,k\,|\,\sigma,\sigma')) \mathbf{1}(\tau = \sigma,\tau'=\sigma'),
    \end{align}
    and we use the notation ${\bf 1}(E)$ for the indicator function of the event $E$. 
If $\cA$ denotes the set of equivalence classes corresponding to the irreducible components of the kernel $\cK$,
it follows that $\cQ_J^{\cK}$ satisfies the detailed balance condition with respect to $\mu_{J,h}\otimes\mu_{J,h}$, for any choice of external fields $h\in\bbR^n$ such that $h$ is constant over every subset $A\in \cA$:
\begin{equation}\label{gxk5}
\mu_{J,h}(\sigma) \mu_{J,h}(\sigma')\cQ_J^{\cK}(\sigma, \sigma'; \tau, \tau') = \mu_{J,h}(\tau) \mu_{J,h}(\tau') \cQ_J^{\cK}(\tau, \tau'; \sigma, \sigma'),
\end{equation}
for arbitrary $(\s,\s'),(\t,\t')\in\Omega\times\Omega$ and $h\in\bbR^n$ such that $h_i=h_j$ for all $i,j\in A$, for all $A\in \cA$; see Lemma \ref{lem:db}. 

In conclusion, for any given interaction $J\in\bbR^{n\times n}$, and any transport kernel $\cK$, we have constructed a probability kernel $\cQ_J^\cK$ on $\Omega\times\Omega$, which satisfies the detailed balance condition with respect to $\mu_{J,h}\otimes\mu_{J,h}$, independently of the choice of $h$, provided $h$ is constant over the irreducible components of $\cK$.

\subsection{Nonlinear dynamics}
Let $\cP(\Omega)$ denote the set of probability measures on $\Omega$.
Given the collision kernel $\cQ_J^\cK$, the nonlinear evolution on $\cP(\Omega)$
is defined by the dynamical system
\begin{equation}
\label{non-linear rqs}
    \frac{dp_t}{dt} = p_t \circ p_t - p_t, \qquad p_0 = p,
\end{equation}
for some initial value \(p \in \mathcal{P}(\Omega)\), where \(p \circ q \in \mathcal{P}(\Omega)\), for \(p, q \in \mathcal{P}(\Omega)\), denotes the commutative collision product defined as
\begin{equation}\label{collision operator}
p \circ q \,(\tau) = \frac{1}{2}\sum_{\sigma, \sigma',\t'\in \O}\big(p(\sigma) q(\sigma')+p(\sigma')q(\sigma)\big)\mathcal{Q}_J^{\cK}(\sigma, \sigma'; \tau, \tau').
\end{equation}
Note that this has the form \eqref{non-linear rqs00} with 
\[
Q(\t\tc\si,\si') = \textstyle{\sum_{\t'\in\O}}\,\mathcal{Q}_J^{\cK}(\sigma, \sigma'; \tau, \tau')\,.
\]
For any $J\in\bbR^{n\times n}$ and any stochastic matrix $\cK$, the Cauchy problem \eqref{non-linear rqs}
has a unique solution $p_t\in \cP(\O)$, $t\ge 0$, for any initial state $p \in \mathcal{P}(\Omega)$. This can be seen by standard existence and uniqueness results, see e.g.\ \cite{BBS} for related statements. Moreover, a  representation of the solution based on random binary  trees can be obtained, see e.g.\ \cite{CCG00,CP2024nonlinear,caputo2025cutoff} for similar constructions. 
Here, we are interested in the structure of the stationary states of \eqref{non-linear rqs}, that is the set of $p \in \mathcal{P}(\Omega)$ such that $p_t\equiv p$, and in the analysis of the convergence to stationarity.  
As we will see, the set $\cA$ of irreducible components of the kernel $\cK$ determines the conservation laws of the dynamical system \eqref{non-linear rqs} and the set of stationary states is fully determined by $J$ and $\cA$. The following are some examples of kernels $\cK$ to keep in mind.

\begin{example}[Single-Site  Dynamics] \label{ex:ss}  
If \(\cK(i,j)={\bf 1}(i=j)\) for all \(i,j\in [n]\), 
exchanges are allowed only between spins at the same site. With the notation \eqref{gxk4}, the 
collision kernel is given by
\begin{equation}
      \mathcal{Q}_J^\mathcal{K}(\sigma, \sigma'; \tau, \tau') = \frac1n\sum_{\ell=1}^n Q_J^{\ell,\ell}(\sigma, \sigma'; \tau, \tau')\,.
\end{equation}
In this case the partition associated to \(\cK\) is made of singletons: \(\cA=(\{i\}, i\in [n])\). The system \eqref{non-linear rqs} is the continuous-time counterpart of the nonlinear Single-Site Dynamics introduced in \cite{AC2024nonlinear}.
\end{example}

\begin{example}[Mean-Field Exchange Dynamics]   \label{ex:mf}
If \(\cK(i,j)=1/n\) for all \(i,j\in[n]\), the proposed exchanges are uniformly chosen between all pairs of sites. 
The 
collision kernel is given by
\begin{equation}
      \mathcal{Q}_J^\mathcal{K}(\sigma, \sigma'; \tau, \tau') = \frac1{n^2}\sum_{\ell,k=1}^n Q_J^{\ell,k}(\sigma, \sigma'; \tau, \tau')\,.
\end{equation}
In this case the partition $\cA$ 
consists  of a single component \(\cA=\{[n]\}\), and the process \eqref{non-linear rqs} is a nonlinear analogue of a mean-field Kawasaki dynamics. 
\end{example}

The two examples discussed above are both covered by the general class of examples defined as follows. 

\begin{definition}[Multi-component Mean-Field Exchange Dynamics]   \label{ex:mcmf}
Fix an arbitrary partition $\cA$ of $[n]$, and define the matrix $\cK=\cK_\cA$ by 
\begin{equation}\label{eq:meanf0}
        \mathcal{K}_\cA(\ell, k) = 
        \begin{cases}
            \frac1{|A|}, & \text{if } \ell, k \in A \text{ for some } A \in \mathcal{A}, \\
            0, & \text{otherwise},
        \end{cases}
    \end{equation}
    In this case, the nonlinear dynamics associated to the kernel $\mathcal{Q}_J^\mathcal{K}$ in \eqref{gxk2} is referred to as the Nonlinear Multi-component Mean-Field Exchange Dynamics.
 \end{definition}
\begin{definition}\label{def:Amag}
Given a partition $\cA$ of $[n]$, and $p \in \mathcal{P}(\Omega)$, we define the $\cA$-\emph{magnetization profile} as the  vector
\begin{equation}\label{eq:Amag}
m_\cA(p) = (m(p,A)\,,\, A\in \cA)\,,\qquad  m(p,A) = \frac1{|A|}\sum_{\ell\in A}p[\si_\ell]\,,
            \end{equation}
            where $p[\si_\ell]$ is the expected value of the spin at site $\ell$ according to $p$.
            The vector $m_\cA(p) $ is in  $[-1,1]^\cA$, and we say that $p$ is $\cA$-{\em regular} if $m_\cA(p)\in(-1,1)^\cA$, that is if $|m(p,A)|\neq 1 $ for all $A\in \cA$. We write $\mathcal{P}_\cA(\Omega)$ for the set of $p \in \mathcal{P}(\Omega)$ that are $\cA$-regular. 
\end{definition}
We note that in Example \ref{ex:ss}, specifying the vector $m_\cA(p)\in[-1,1]^n$ is
equivalent to giving all the single-site marginals of $p$, while in Example \ref{ex:mf}, $m_\cA(p)\in[-1,1]$ is just the global mean magnetization of $p$.  

The fist result is the following general convergence theorem, which  summarizes the main properties of our nonlinear dynamical system.
We use the notation 
$\G_\cA\subset \bbR^n$ for the set of external fields that are constant over each element of the partition $\cA$: 
\[
\G_\cA = \{h\in\bbR^n:\; h_i=h_j,  \; \forall i,j\in A\,,\; \forall A\in\cA\}
.\]
It can be shown that for each $p\in\mathcal{P}_\cA(\Omega)$ there is a  unique vector $h\in\G_\cA$ depending only on $J$ and $m_\cA(p) $, such that  
$m_\cA(p) = m_\cA(\mu_{J,h})$; see Lemma \ref{lem:uni}. 
\begin{theorem}\label{th:basic}
Fix a symmetric matrix $J\in\bbR^{n\times n}$ and a symmetric stochastic $n\times n$ matrix $\cK$. Let $\cA$ denote the partition of $[n]$ induced by the irreducible components of $\cK$. 
 A probability $p \in \mathcal{P}_\cA(\Omega)$ is a stationary state if and only if $p=\mu_{J,h}$ for some $h\in\G_\cA$.
 Moreover, 
 the expected values
$m(p_t,A)$, $A\in \cA$,
are independent of $t\ge 0$. Finally, for any initial state $p\in \mathcal{P}_\cA(\Omega)$, one has the convergence 
\begin{equation}
\label{eq:conv}
\|p_t - \mu_{J,h}\|_{\rm TV}\; \longrightarrow \;0\,,\qquad t\to\infty\,,
\end{equation}
where $\|\cdot\|_{\rm TV}$ denotes the total variation distance and 
$h$ is the unique vector $h\in\G_\cA$ such that 
$m_\cA(p) = m_\cA(\mu_{J,h})$. 
\end{theorem}
\begin{remark}\label{rem:hinfinity}
We note that the requirement that $p$ be $\cA$-regular in Theorem \ref{th:basic} is not essential, as it can be removed if one allows the external fields $h$ to take the value $\pm\infty$. 
\end{remark}
\subsection{Entropy decay}
The above theorem establishes convergence to stationarity, but it says nothing about the speed of the convergence. The next results address this question, under a high-temperature assumption on the interaction matrix $J$. Moreover,  we restrict to the case of Multi-component  Mean Field Dynamics, that is we take  $\cK=\cK_\cA$ as in Definition \ref{ex:mcmf} with an arbitrary choice of the partition $\cA$.

Given $\nu,\mu\in\cP(\O)$, we write $H(\nu\,|\,\mu)$ for the relative entropy, or KL-divergence,  of $\nu$ with respect to $\mu$, that is  
\begin{equation}\label{relative entropy1}
    H(\nu\,|\,\mu) =
\begin{cases}
    \sum_{\si\in\O}\nu(\si)\log(\nu(\si)/\mu(\si)) & \text{if} \ \ \nu \ll \mu, \\
    +\infty & \text{otherwise},
\end{cases}
\end{equation}
where $\nu\ll \mu$ denotes absolute continuity of $\nu$ with respect to $\mu$.

Our main quantitative result is formulated as follows.   
 \begin{theorem}\label{Central Theorem} 
Let \(J\in\bbR^{n\times n}\) be a nonnegative definite matrix with largest eigenvalue \(\lambda(J) < \frac{1}{2}\), and let  \(\overline{J}=\max_{i\in [n]} \sum_{j\in [n]}|J_{i,j}|\). 
Then, 
for any partition $\cA$ of $[n]$, the nonlinear system \eqref{non-linear rqs} defined by $J$ and $\cK$, with $\cK=\cK_\cA$ given in \eqref{eq:meanf0}, satisfies 
\begin{equation}\label{relative entropy01}
    H(p_t\,|\,\mu_{J,h}) \le H(p\,|\,\mu_{J,h})\,e^{-\a\, t}\,,\qquad t\ge 0, 
\end{equation}
for any $p\in \mathcal{P}_\cA(\Omega)$, where $h$ is the unique vector $h\in\G_\cA$ such that 
$m_\cA(p) = m_\cA(\mu_{J,h})$, and the positive constant $\a$ satisfies 
  \begin{equation}
        \alpha \geq \tfrac1{4n}(1 - 2\lambda(J))^2\,e^{-16\overline{J}}
\,.        
   \end{equation}
\end{theorem}

\begin{remark}\label{rem:pinsk}
We note that the estimate \eqref{relative entropy01} implies exponential decay in total variation as well. Indeed,  by Pinsker's inequality, 
$\|p_t-\mu_{J,h}\|^2_{\rm TV} \le\frac12 H(p_t\,|\,\mu_{J,h})$. Using $H(p\,|\,\mu_{J,h})\le \max_{\si\in\O}\log (1/\mu_{J,h}(\si))$, which can be bounded by $Cn$, where $
C=\l(J) + 2\overline{h}+\log(2)$ and $\overline{h}=\max_{i\in [n]} |h_i|$, it follows that, in the setting of Theorem \ref{Central Theorem}, for all $p\in\cP_\cA(\O)$, one has
  \begin{equation}\label{eq:tvbou}
 \|p_t-\mu_{J,h}\|_{\rm TV} \le  \,\sqrt{\frac{C\, n}2}\,e^{-\a\, t/2}\,.
 \end{equation}
 \end{remark}
In particular, Theorem \ref{Central Theorem} is a quantitative version of Theorem \ref{th:basic} under the additional assumption on the spectral radius of $J$, which can be regarded as a high temperature requirement. It is worthy of note that the rate $\a$ is bounded below independently of the choice of partition $\cA$. In the single-site case of Example \ref{ex:ss} this reproduces essentially the same rate that was obtained,   under a more restrictive high temperature constraint, for the discrete-time version of the single-site model in \cite{AC2024nonlinear}.

By comparison with the non-interacting case $J=0$, we note that the rate $\a$ has the optimal dependence with respect to the dimension $n$, up to a constant factor. For instance, in the single-site case, one can use a coupon collector argument to show that if $J=0$ and the initial distribution $p$ is all $+1$ with probability $1/2$  and all $-1$ with probability $1/2$, then the left hand side in \eqref{eq:tvbou} is at least $c\,e^{-\frac{t}{c\,n}}$, $t\ge 0$, for some absolute constant $c>0$; see \cite{AC2024nonlinear}. 


The estimate in Theorem \ref{Central Theorem} 
is equivalent to a certain nonlinear logarithmic Sobolev inequality, which provides a lower bound on the relative entropy dissipation in terms of the relative entropy; see Section \ref{sec:nmlsi} below. 
Rather than establishing such a bound directly, to prove Theorem \ref{Central Theorem} we follow the strategy introduced by Mark Kac in his seminal 1956 paper \cite{KAC}, that is we 
approximate the nonlinear evolution of one particle by a linear mean-field type Markov
process involving a large number of particles.

\subsection{Kac's program}\label{subsec:kac}
We consider a system of \(N \in \mathbb{N}\) {\em particles} \(\{\sigma(i), i = 1, \ldots, N\}\), where each \(\sigma(i)=(\si_1(i),\dots,\si_n(i))\) is a spin configuration  in \(\Omega = \{-1, +1\}^n\). 
In words, the process is described as follows. Independently, each ordered pair of particles $(i,j)$, including $i=j$, is updated at the arrival times of a Poisson process with rate $1/N$. When an update occurs at particles $(i,j)$, a site $\ell$ is picked uniformly at random,  a site $k$ is picked according to the transport kernel $\cK(\ell,k)$, and then with probability $\hat p_J(\ell,k\tc\si(i),\si(j))$ the $\ell$-th spin of the $i$-th particle is exchanged with the $k$-th spin of the $j$-th particle, 
while with probability $1- \hat p_J(\ell,k\tc\si(i),\si(j))$ the particle configuration is  left unchanged. 
Using the notation \eqref{gxk1}, the exchange can be written as
\[
(\si(i),\si(j))\mapsto (S_\ell(\sigma(i), \sigma(j)_k)\,,\,S_k(\sigma(j), \sigma(i)_\ell),
\]
with the understanding that when $i= j$ this is replaced by  $\si(i)\mapsto \si(i)^{\ell,k}$, where $\si(i)^{\ell,k}$ denotes the configuration $\si(i)$ after the exchange $(\si_\ell(i),\si_k(i))\mapsto(\si_k(i),\si_\ell(i))$. Here the probability 
 $\hat p_J(\ell,k\tc\si(i),\si(j))$ is defined by 
 \begin{equation}\label{hatp}
\hat p_J(\ell,k\tc\si(i),\si(j))= \begin{cases} p_J(\ell,k\tc\si(i),\si(j)) & i\neq j\\
 \frac{ \mu_J(\si(i)^{\ell k})}{  \mu_J(\si(i))+  \mu_J(\si(i)^{\ell k})} & i= j
\end{cases}
\end{equation}
 where we use the notation \eqref{gxk3}. 
This defines a continuous-time Markov chain with state space $\O^N=\{-1,+1\}^{Nn}$, which is described by the infinitesimal generator 
%
%
%
\begin{equation}\label{particles system generator}
\mathcal{L}_N F(\sigma) = \frac{1}{Nn} \sum_{i,j=1}^N \sum_{\ell,k \in [n]} \cK(\ell,k) \,\hat p_J(\ell,k\tc\si(i),\si(j)) \,\nabla^{i,\ell;j,k} F(\sigma),
\end{equation}
where $\si\in\O^N$,  \(F: \Omega^N \to \mathbb{R}\), and the gradient operator \(\nabla^{i,\ell;j,k}\) is defined by
\begin{equation}\label{eq:unde}
\nabla^{i,\ell;j,k}F(\sigma) := F(\sigma^{i,\ell;j,k}) - F(\sigma), \quad \text{with } \sigma^{i,\ell;j,k}(u) := \begin{cases}
\sigma(u) & u \neq i,j, \\
S_\ell(\sigma(i), \sigma(j)_k) & u = i, \\
S_k(\sigma(j), \sigma(i)_\ell) & u = j,
\end{cases}
\end{equation}
with the convention that, when  $i= j$, then $\sigma^{i,\ell;i,k}(u)=\sigma(u)$ if $ u \neq i$, and $\sigma^{i,\ell;i,k}(i)=\si(i)^{\ell,k}$.

We may picture the system as a spin exchange dynamics in a $N\times n$ grid, with exchanges occurring between sites $(i,\ell)$ and $(j,k)$. By construction, the Markov chain is reversible with respect to the product measure $\mu_J^{\otimes N}$. 
Since only pairs $(\ell,k)$ belonging to sites within the same equivalence class are exchanged,  to turn the Markov chain into an irreducible process we need to impose a constraint of fixed magnetization in each block $B_{N,A}:=\{(i,\ell)\,,\; i\in[N], \ell\in A\}$, $A\in\cA$. To this end, we need the following definitions. 
\begin{definition}\label{def:NAmag}
Given a partition $\cA$ of $[n]$, we write $D_{\cA,N}$ for the set of all vectors $\r=(\r(A),\,A\in\cA)\in[0,1]^\cA$ such that $N|A|\r(A)$ 
is a nonnegative integer. Given $\cA$,  $D_{\cA,N}$ is called the set of $\cA$-\emph{density profiles}. Given $\r\in D_{\cA,N}$, we define
\begin{equation}\label{eq:AmagO}
\O_{N,\r}= \left\{\si\in\O^N: \; \textstyle{\sum_{i\in [N]}\sum_{\ell\in A}}\si_\ell(i)=(2\r(A)-1)N|A|,\; \;\forall A\in\cA\right\}\,.
            \end{equation}
\end{definition}
It follows that the Markov chain with generator \eqref{particles system generator} is irreducible in the state space $\O_{N,\r}$, for every choice of $\r\in D_{\cA,N}$. Moreover, the probability measure 
\begin{equation}\label{eq:pbm}
\mu_{N,J,\r}:=\mu_J^{\otimes N}(\cdot\tc\O_{N,\r})
\end{equation} 
obtained as tensor product of $\mu_J$ restricted to $\O_{N,\r}$ is the associated unique invariant distribution.  We refer to $\mu_{N,J,\r}$ as the multi-canonical Ising measure. 

Our main result concerning the particle system is the following modified log-Sobolev inequality (MLSI), with bounds that are uniform in $N$ and in the density profile $\r\in D_{\cA,N}$. 
To introduce the relevant quantities, observe that the Dirichlet form associated to the reversible pair  $(\cL_N,\mu_{N,J,\r})$ is given by 
\begin{equation}\label{Dirichlet form}
\mathcal{E}_{N, J,\r}(F, G) := \frac{1}{2Nn} \sum_{i,j=1}^N \sum_{\ell,k =1}^n \cK(\ell,k)\mu_{N,J,\r}\left[ \hat p_J(\ell,k\tc \sigma(i), \sigma(j)) \nabla^{i,\ell;j,k}F\,\nabla^{i,\ell;j,k}G \right],
\end{equation}
for functions \(F, G : \O_{N,\r} \to \mathbb{R}\). Recall  that the {\em entropy} of a function $F\ge 0$ with respect to a probability measure $\mu$  is given by 
\(\operatorname{Ent}_{\mu}F := \mu[F \log F] - \mu[F] \log \mu[F]\).

\begin{theorem}[MLSI for the particle system]\label{th:MLSI}
Let \(J\in\bbR^{n\times n}\) be a nonnegative definite matrix with largest eigenvalue \(\lambda(J) < 1/2\), and let  \(\overline{J}=\max_{i\in [n]} \sum_{j\in [n]}|J_{i,j}|\). 
Then, 
for any partition $\cA$ of $[n]$, for any $N\ge 2$, and $\r\in \cD_{\cA,N}$, the Markov chain with Dirichlet form \eqref{Dirichlet form} defined by $J$ and $\cK=\cK_\cA$ given in \eqref{eq:meanf0},  satisfies 
\begin{equation}\label{eq:MLSI1}
\mathcal{E}_{N, J,\r}(F, \log F) \ge \a\operatorname{Ent}_{\mu_{N,J,\r}}F,
\end{equation}
for all $F\ge 0$, with constant  $\a$ such that  
  \begin{equation}\label{eq:MLSI2}
        \alpha \geq \tfrac1{4n}(1 - 2\lambda(J))^2\,e^{-16\overline{J}}\,.
   \end{equation}
\end{theorem}
It is well known that a MLSI of the form \eqref{eq:MLSI2} is equivalent to an exponential decay statement for the relative entropy, with decay rate given by the same constant $\a$. That is, for any initial distribution $\nu$ on $ \O_{N,\r}$, the distribution $\nu_t:=\nu e^{t\cL_N}$ of the particle system at time $t$ satisfies 
    \begin{equation}\label{eq:MLSI3}
      H_N( \nu_t\tc  \mu_{N,J,\r}) \le H_N( \nu \tc  \mu_{N,J,\r})\,e^{- \alpha\,t},\qquad t\ge 0, 
   \end{equation}
   where $ H_N$ denotes the relative entropy for distributions on $\O^N$.  
The key feature of the bound in Theorem \ref{th:MLSI} is the independence on $N$ and $\r\in \cD_{\cA,N}$. This, together with so-called Entropic Chaos and Fisher Chaos results for our system, will allow us to transfer the result \eqref{eq:MLSI3} on the linear particle system to the desired entropy decay estimate \eqref{relative entropy01} for the nonlinear system, as stated in Theorem \ref{Central Theorem}; see Section \ref{sec:kac}.

\subsection{Related work}

Boltzmann-type equations of the form~\eqref{non-linear rqs0} have been extensively studied across a wide range of fields.  
Besides being a classical topic in kinetic theory~\cite{Vill1}, they have been systematically investigated in the framework of \emph{mass-action kinetics}; see~\cite{FeinbergBook} for a recent survey.  
Moreover, such equations capture a variety of phenomena, from models in population genetics and genetic algorithms in combinatorial optimization~\cite{goldberg,Mitchell}, to random sampling~\cite{Sinetal,Sinetal2,SV13}.  
The sampling perspective is particularly relevant in our setting, since, as shown in Theorem~\ref{th:basic}, the nonlinear dynamics can be used to approximate a system with a prescribed interaction structure (determined by the matrix~$J$) and a specified magnetization profile (determined by the initial condition and the partition~$\mathcal{A}$).
We note that these types of dynamical systems can also be formulated within the broader framework of \emph{nonlinear Markov chains}, introduced in~\cite{McKeanNLMC}, which provides a convenient setup for the analysis of mean-field models.  


The specific spin-system framework adopted in this paper was introduced under the name of \emph{Reversible Quadratic Systems} in~\cite{AC2018entropy}, where 
the decay of relative entropy was analyzed for a class of recombination models at infinite temperature.  
Further developments on nonlinear recombination models were obtained in~\cite{CP2024nonlinear,caputo2025cutoff}.  
More recently,~\cite{AC2024nonlinear} established the first quantitative results for an interacting quadratic system by proving exponential decay of the total variation distance at high temperature, for a model in which exchanges occur at coinciding sites in the two spin configurations.  
This yields a process conserving the full single-site magnetization profile.  
In the present work, we considerably strengthen the results of~\cite{AC2024nonlinear} by allowing arbitrary magnetization profiles, and by upgrading the convergence to exponential decay in relative entropy. 


Proving exponential decay of relative entropy for kinetic equations is a notoriously difficult problem; see~\cite{Vill2,entfound,kacsprog,Bonetto2018,Max} and references therein.  
The investigation of this question has driven substantial advances in the field, with the concept of \emph{Propagation of Chaos}, introduced by Kac in his foundational paper~\cite{KAC}, playing a central role.  
We refer to~\cite{kacsprog} for a detailed analysis of Kac's program in the context of the Boltzmann equation, and to~\cite{revchaos} for a recent broad survey.  


As shown in Theorem~\ref{th:MLSI}, our analysis relies on proving that the relative entropy of the corresponding Kac particle system decays exponentially, with rates that are uniform in both the number of particles $N$ and the conserved quantities $\r$.  
Equivalently, this amounts to establishing a uniform \emph{Modified Logarithmic Sobolev Inequality} (MLSI).  
The validity of such a strong functional inequality may come as a surprise, as it stands in sharp contrast to the classical Kac-Boltzmann setting; see~\cite{Vill2,entfound,kacsprog,Bonetto2018} for related discussions.  

Functional inequalities for conservative spin systems have been widely studied; see~\cite{Yau,MCR} for the case of lattice spin systems under a spatial mixing condition.  
Although these results hold uniformly in the conserved parameters $\r$, the techniques developed in those works do not provide the uniformity in the number of particles $N$ required here.  
To prove Theorem~\ref{th:MLSI}, we instead adapt the framework of \emph{stochastic localization} to our setting.  
We refer to~\cite{chen2022localization} for an introduction to stochastic localization, and to~\cite{bauerschmidt2024polchinsky} for an alternative, essentially equivalent formulation based on renormalization group ideas.  
In particular, inspired by the recent developments in Kawasaki dynamics~\cite{bauerschmidt2024kawasaki}, we exploit a reduction to the case of independent spins with fixed magnetization and arbitrary external fields, for which a modified log-Sobolev inequality follows from a bound on the so-called \emph{Down-Up walk} on matroid bases.  
The latter, in turn, rests on important recent advances concerning the entropy of particle systems with log-concave generating polynomials~\cite{CGM,EntInd22}.  
Our approach follows closely the formulation proposed in~\cite{caputo2025factorizations}, which is based on the \emph{entropy factorization} inequality rather than the modified log-Sobolev inequality.  
This perspective is crucial for handling arbitrary magnetization profiles, as it allows us to obtain flexible factorizations of the multi-canonical measure~\eqref{eq:pbm}.

\subsection{Plan of the paper}
General structural properties of the dynamical system are discussed in Section~\ref{sec:struct}, which also contains  the proof of Theorem~\ref{th:basic}.  
The analysis of the associated Kac particle system, together with the proof of Theorem~\ref{th:MLSI}, is presented in Section~\ref{sec:MLSI}.  
Finally, in Section~\ref{sec:kac} we show how to transfer these results from the Kac particle system to the nonlinear equation, thereby completing the proof of Theorem~\ref{Central Theorem}. We conclude with some open problems about establishing exponential decay to equilibrium without the high temperature assumption.

\section{Structural properties of the nonlinear equation}\label{sec:struct}
The statements in this section hold for any integer $n\in\bbN$, for any fixed choice of the symmetric matrix  $J\in\bbR^{n\times  n}$ and the symmetric stochastic $n\times n$ matrix $\cK$.  

\subsection{Preliminary facts}
Recall the notation $\cA$ for the partition of $[n]$ induced by the irreducible components of  $\cK$ and the set $\G_\cA$ of external fields associated to $\cA$. 
\begin{lemma}\label{lem:db}
The kernel $\cQ_J^\cK$ defined by \eqref{gxk2} satisfies 
\begin{equation}\label{db5}
\mu_{J,h}(\sigma) \mu_{J,h}(\sigma')Q_J^{\cK}(\sigma, \sigma'; \tau, \tau') = \mu_{J,h}(\tau) \mu_{J,h}(\tau') Q_J^{\cK}(\tau, \tau'; \sigma, \sigma'),
\end{equation}
for all $h\in \G_\cA$, and $\si,\si',\t,\t'\in\O$.
\end{lemma}
\begin{proof}
Let $p_{J,h}(\ell,k|\si,\si')$ denote the quantity defined in \eqref{gxk3} with the difference that $\mu_J$ is replaced by $\mu_{J,h}$, with $h\in\bbR^n$.  By the definition of $\G_\cA$, it follows that, whenever $\ell,k$ belong to the same irreducible class,  
\[
p_{J}(\ell,k|\si,\si') = p_{J,h}(\ell,k|\si,\si')\,, \qquad h\in\G_\cA.
\]
In particular, if $(\t,\t')=\left(S_\ell(\sigma, \sigma'_k),\ S_k(\sigma', \sigma_\ell)\right)$, and $ h\in\G_\cA$, then 
\[
\mu_{J,h}(\si)\mu_{J,h}(\si')p_{J}(\ell,k|\si,\si') = \mu_{J,h}(\t)\mu_{J,h}(\t')p_{J}(\ell,k|\t,\t').
\]
Since only pairs $(\ell,k)$ such that $\cK(\ell,k)>0$ enter the definition \eqref{gxk2} of $\cQ_J^\cK$,  this shows that the kernel $\cQ_J^{\cK}$ satisfies the claimed detailed balance \eqref{db5}.
\end{proof}
Recall the Definition \ref{def:Amag} of the magnetization profile $m_\cA(p)$.
\begin{lemma}\label{lem:uni}
Fix a partition $\cA$ of $[n]$. For any vector $v = (v_A,\,A\in\cA)$, with $v_A\in(-1,1)$, for all $A\in\cA$, there exists a unique choice of $h\in\G_{\cA}$ such that 
$m_\cA(\mu_{J,h})= v$. 
\end{lemma}
\begin{proof}
We adapt 
the argument from Lemma~1 in \cite{BubeckEldan}. 
We need to show that the map 
\[
h \mapsto m_\cA(\mu_{J,h})
\]
is a bijection between $\G_\cA$ and $(-1,1)^\cA$. Using the change of variables $\theta_A:=\sum_{i\in A}h_i$, $A\in\cA$, this is equivalent to showing that  the map 
\begin{equation}
\theta 
\in\bbR^{\cA}
\;\;\longmapsto\;\; f(\theta) := \nabla \varphi(\theta) \in (-1,1)^\cA
\end{equation}
is a bijection, where 
\[
\varphi(\theta) = \log \left(\sum_{\si\in\O} e^{\sum_{A \in \cA} \theta_A s_A(\sigma)}
\mu_J(\sigma)\right)\,,\qquad s_A (\si):= \frac{1}{|A|}\sum_{i \in A}\sigma_i\,.
\]
The Jacobian \(\nabla f = \nabla^2 \varphi \) coincides with the covariance matrix of the random vector \((s_A)_{A \in \cA}\) 
when \(\sigma\) is distributed according to \(\mu_{J,h}\), which is easily seen to be strictly positive definite for all $\theta\in\bbR^\cA$.  Therefore, $\varphi$ is strictly convex and $f=\nabla\varphi$ is injective.
To show that it is also surjective,   
it is sufficient to show that the image of $f$,  $\mathrm{Im}(f) $, satisfies 
\[
\partial \mathrm{Im}(f) \subset \partial (-1,1)^\cA.
\]
Indeed, suppose by contradiction that there exists \(y \in (-1,1)^\cA \setminus \mathrm{Im}(f)\), and consider the line segment \([f(0),y]\). Since \(\mathrm{Im}(f)\) is open, this segment must intersect \(\partial \mathrm{Im}(f)\) at some point \(z\). If $\partial \mathrm{Im}(f) \subset \partial (-1,1)^\cA$, then $z \in\partial (-1,1)^\cA$, which contradicts the fact that \((-1,1)^\cA\) is convex.

To prove $\partial \mathrm{Im}(f) \subset \partial (-1,1)^\cA$, take \(w \in \partial \mathrm{Im}(f)\). Then there exists a sequence \((\theta_k) \subset \mathbb{R}^\cA\) such that 
$\lim_{k \to \infty} f(\theta_k) = w$.
By continuity of \(f\), the sequence \((\theta_k)\) must be unbounded, since otherwise one could extract a converging subsequence $\theta'_k\to \theta\in\bbR^\cA$ which would imply that \(w=f(\theta) \in \mathrm{Im}(f)\).  
Therefore, there exists a subsequence \((\theta'_{k})\) and some \(A \in \cA\) such that \((\theta'_{k})_A \to + \infty\) or \((\theta'_{k})_A \to - \infty\).  
It follows that
\[
w_A = \lim_{j \to \infty} f(\theta'_{k})_A \in \{-1,1\}, 
\]
and therefore $w\in\partial (-1,1)^\cA$.
\end{proof}

\subsection{Entropy and stationary states}
A distribution \(p \in \mathcal{P}(\Omega)\) is called {\em stationary} if  $p\,\circ\, p=p$, where the collision product ``$\circ$'' is defined in \eqref{collision operator}. Equivalently, $p$ is stationary iff $p_t=p$, for all $t\ge0$.  It follows from the reversibility in Lemma \ref{lem:db} that $\mu_{J,h}\circ \mu_{J,h}=\mu_{J,h}$, for any $h\in\G_\cA$, and thus all measures $\mu_{J,h}$, $h\in\G_\cA$, are stationary. To characterize all stationary states we use the relative entropy functional.

 The following statement is a version of Boltzmann's H-theorem from kinetic theory. 
\begin{lemma}\label{H-theorem}  
Let $p_t$ denote the solution of (\ref{non-linear rqs}) with initial datum $p\in\cP(\O)$, and let $\mu = \mu_{J,h}$ for some arbitrary, fixed  $h\in\G_\cA$. Then,  
%
\begin{equation} \label{H-theorema}
    \frac{d}{dt} H(p_t\tc\mu) = -\mathcal{D}_{\mu}(f_t), 
    \,\qquad t\ge0, 
\end{equation}
where  $f_t: = p_t/\mu$, 
and the entropy dissipation functional \(\mathcal{D}_{\mu}\) is  defined by 
\begin{equation} \label{dissipation energy nonlin}  
   \mathcal{D}_{\mu}(f)  = \frac{1}{4} \sum_{\sigma, \sigma', \tau, \tau'}\mu(\tau) \mu(\tau') \mathcal{Q}_J^\cK(\tau, \tau'; \sigma, \sigma') \left(f(\tau)f(\tau') - f(\sigma)f(\sigma')\right) \log\left(\frac{f(\tau)f(\tau')}{f(\sigma)f(\sigma')}\right).
\end{equation}
In particular,  $H(p_t\tc \mu)$ is non-increasing in $t\ge 0$. Moreover, $p\in \mathcal{P}(\Omega)$ is stationary iff the density $f=p/\mu$ satisfies $ f(\tau)f(\tau') = f(\sigma)f(\sigma')$ whenever $\mathcal{Q}_J^\cK(\tau, \tau'; \sigma, \sigma')>0$. 
 \end{lemma}
\begin{proof}
For any $p\in \mathcal{P}(\Omega)$, setting $f=p/\mu$, and using Lemma \ref{lem:db}, for any $\t\in\O$,
\begin{equation} \label{eq:rqs0}
p\circ p\,(\t) - p(\t)=   \sum_{\sigma, \sigma', \tau'}\mu(\si) \mu(\si') \mathcal{Q}_J^\cK(\sigma, \sigma';\tau, \tau') \left( f(\sigma)f(\sigma')-f(\tau)f(\tau') \right). 
\end{equation}
Therefore, differentiating the relative entropy,  
\begin{align} \label{H-theorema01}
    \frac{d}{dt} H(p_t\tc\mu) &=  \sum_{\t}\left[\frac{d}{dt} \,p_t(\t) \right]\log \frac{p_t(\t) }{\mu(\t) }
    \\ & =- \sum_{\sigma, \sigma', \tau, \tau'}\mu(\tau) \mu(\tau') \mathcal{Q}_J^\cK(\tau, \tau'; \sigma, \sigma') \left(f_t(\tau)f_t(\tau') - f_t(\sigma)f_t(\sigma')\right) \log f_t(\tau).
\end{align}
Symmetrizing first with respect to  $\t \leftrightarrow \t'$ and then with respect to $(\t,\t') \leftrightarrow (\si,\si')$, Lemma \ref{lem:db} shows that the above expression can be rewritten as \eqref{H-theorema}.
Note that \(\mathcal{D}_{\mu}(f)\ge 0\), for any function $f:\O\to \bbR_+$,  since every summand in its definition is nonnegative. Now, if $p$ is stationary, then  the density $f=p/\mu$ satisfies $\cD_\mu(f)=0$, and therefore, by positivity of $\mu$, $ f(\tau)f(\tau') = f(\sigma)f(\sigma')$ whenever $\mathcal{Q}_J^\cK(\tau, \tau'; \sigma, \sigma')>0$. 
Conversely, if $ f(\tau)f(\tau') = f(\sigma)f(\sigma')$ whenever $\mathcal{Q}_J^\cK(\tau, \tau'; \sigma, \sigma')>0$, then \eqref{eq:rqs0} shows  that $p\circ p=p$. 
\end{proof}
%

We turn to a characterization of the stationary states. Let $\mathcal{P}_+(\Omega)$ denote the set of positive distributions, that is the set of \(p \in \mathcal{P}(\Omega)\) such that \(p(\sigma) > 0\) for all \(\sigma \in \Omega\). Clearly, $\mu_{J,h}\in\mathcal{P}_+(\Omega)$ for any $h\in\bbR^n$. 
 
\begin{lemma}[Stationary measures] \label{charact-of-stationarity}
A probability measure \(p \in \mathcal{P}_+(\Omega)\)
is stationary  if and only if \(p =\mu_{J,h}\) for some $h\in\G_\cA$.
\end{lemma}
\begin{proof}
We already know that $\mu_{J,h}$ is stationary, for any $h\in\G_\cA$. It remains to show that these are the only positive stationary distributions. 
Fix $\mu=\mu_{J,h}$ for some $h\in\G_\cA$, and let \(p \in \mathcal{P}_+(\Omega)\)
be a stationary measure. From Lemma \ref{H-theorem}, $f=p/\mu$ satisfies
\begin{equation}
\label{revmu1}
f(\sigma)f(\sigma')= f(S_\ell(\sigma, \sigma'_k))f(S_k(\sigma', \sigma_\ell)),
\end{equation}
for all  \(\ell, k \in [n]\) such that \(\mathcal{K}(\ell, k) > 0\), for all \(\sigma, \sigma' \in \Omega\).
For any $\ell$, fix some $k=k(\ell)\in[n]$ such that \(\mathcal{K}(\ell, k) > 0\). 
For any \(\sigma \in \Omega\) and any \(j \in [n]\), let \(\sigma[j] \in \Omega\) denote the configuration equal to \(+1\) on all sites \(\{1, \dots, j\}\) and equal to \(\sigma\) on \(\{j+1, \dots, n\}\). Setting \(\sigma[0] = \sigma\), we have:
\[
  f(\sigma) = f(+) \prod_{j=1}^n \frac{f(\sigma[j-1])}{f(\sigma[j])},
\]
where $\si[n]=+$ is the all $+$ configuration.
Define \(\varphi_j(a) := \frac{f(+^{k(j), a})}{f(+)}\), where \(+^{k,a}\) denotes the configuration equal to \(a\) at site \(k\) and \(+1\) elsewhere. 
Taking $\si=\si[j-1]$, $\si'=+$, from \eqref{revmu1} with $j=\ell$ and $k=k(j)$, 
\[
  \frac{f(\sigma[j-1])}{f(\sigma[j])} = \frac{f(+^{k(j), \sigma_j})}{f(+)} = \varphi_j(\sigma_j).
\]
Since \(\varphi_j\) is a positive function on \(\{-1, +1\}\), we can write \(\varphi_j(a) = c_j e^{h_j a}\), \(a = \pm 1\), for some constants \(c_j>0,h_j \in \mathbb{R}\). It follows that
\[
f(\sigma) =C\exp\left(\textstyle{\sum_{j \in [n]}} h_j\sigma_j \right),
\]
for some $h\in\bbR^n$, where $C>0$ is a constant.
It remains to prove that $h\in\G_\cA$. To see this, note that by \eqref{revmu1}, whenever $\cK(\ell,k)>0$ one has
\begin{equation}
\label{revmu2}
(h_\ell-h_k)(\si'_k-\si_\ell) = 0\,,
\end{equation}
for all $\si,\si'$. It follows that $h_\ell=h_k$ if $\cK(\ell,k)>0$. Therefore $h\in\G_\cA$.
\end{proof}

\subsection{Conservation laws}
Recall the Definition \ref{def:Amag} of the magnetization profile $m_\cA(p) = (m(p,A), A\in\cA)$. 

\begin{lemma}\label{conserved magn}
For all \(p, q \in \mathcal{P}(\Omega)\), 
\begin{equation}
\label{magn1}
  m_{\mathcal{A}}(p \circ q) = \tfrac{1}{2}\,m_{\mathcal{A}}(p )+ \tfrac{1}{2}\,m_{\mathcal{A}}(q).
\end{equation}
In particular, if \(p_t\) is the solution of \eqref{non-linear rqs} with initial datum \(p\), then
$ \frac{d}{dt} m_\cA(p_t)= 0$ 
%
\end{lemma}

\begin{proof}
By  definition, 
\[
 m(p \circ q,A) 
 = \frac{1}{|A|} \sum_{i \in A} p \circ q \,[\t_i],
\]
where 
\[
p \circ q \,[\t_i]=
 \frac{1}{2} \sum_{\sigma, \sigma'}  p(\sigma) q(\sigma') \sum_{\tau, \tau'} \left(\mathcal{Q}^{\mathcal{K}}_J(\sigma, \sigma'; \tau, \tau') + \mathcal{Q}^{\mathcal{K}}_J(\sigma', \sigma; \tau, \tau')\right) \tau_i.
\]
Moreover,
\[
\sum_{\tau, \tau'} \mathcal{Q}^{\mathcal{K}}_J(\sigma, \sigma'; \tau, \tau') \,\tau_i
= \si_i  + \frac1n \sum_{ k=1}^n \cK(i,k) p_J(i, k \tc \sigma, \sigma') (\sigma'_k-\sigma_i).
\]
Therefore, 
\[
\frac12\sum_{\tau, \tau'} \left(\mathcal{Q}^{\mathcal{K}}_J(\sigma, \sigma'; \tau, \tau') + \mathcal{Q}^{\mathcal{K}}_J(\sigma', \sigma; \tau, \tau')\right)\tau_i
= \frac12(\si_i+\si'_i)  + \frac1{2n} \sum_{ k=1}^nR_{i,k}(\si,\si')
\]
where
\[
R_{i,k}(\si,\si')=\cK(i,k) \left[p_J(i, k \tc \sigma, \sigma') (\sigma'_k-\sigma_i) - p_J(k,i \tc \sigma, \sigma') (\sigma'_i-\sigma_k) \right],
\]
and we have used the symmetries 
$p_J(i, k \tc \sigma, \sigma')=p_J(k,i \tc \sigma', \sigma)$.
Since $\cK(i,k)=\cK(k,i)$,  $R_{i,k}(\si,\si')$ is antisymmetric in the exchange $i\leftrightarrow k$. Therefore, for any fixed $A\subset [n]$, 
\[
\sum_{i,k\in A}R_{i,k}(\si,\si')=0.
\]
In particular, using that $\cK(i,k)=0$ when $i,k$ do not belong to the same irreducible component, this proves that for any fixed $A\in\cA$,
 \[
\frac1{|A|}\sum_{i\in A} p \circ q \,[\t_i]=\frac12 m(p ,A) + \frac12 m(q ,A)\,. 
\]
\end{proof}

\subsection{Proof of Theorem \ref{th:basic}}
The key to the proof of Theorem \ref{th:basic}
is the following statement about positivity of solutions, that we refer to as the irreducibility lemma. Recall the notation $\cP_\cA(\O)$ for the set of $\cA$-regular distributions. 
\begin{lemma}[Irreducibility]\label{lem:irre}
For any $p\in\cP_\cA(\O)$, there exists \(t_0>0\) and \(\varepsilon_0>0\)  such that for all \(t\geq t_0\) the solution \(p_t\) of \eqref{non-linear rqs} satisfies 
\begin{equation} \label{eq:irreb}
p_t( \tau)\geq \varepsilon_0, \qquad \forall\tau\in\Omega.
\end{equation}
\end{lemma}
The proof requires several steps and is given  in the next subsections. Here we show how to prove Theorem \ref{th:basic}
 once we have Lemma \ref{lem:irre}.
If $p\in\cP_\cA(\O)$, then Lemma \ref{lem:irre} ensures that $p_t\in\cP_+(\O)$ eventually. Therefore, if $p$ is $\cA$-regular and stationary, then it must be in $\cP_+(\O)$, and therefore by Lemma \ref{charact-of-stationarity}, $p=\mu_{J,h}$ for some $h\in\G_\cA$. Next, we show that if $p$ is $\cA$-regular, then $p_t$ converges as $t\to\infty$ to a stationary state $\nu$. If this is the case, then by the above discussion $\nu=\mu_{J,h}$ for some $h\in\G_\cA$, and, using Lemma \ref{lem:uni} and Lemma \ref{conserved magn}, $h$ is the unique element of $\G_\cA$ such that $m_\cA(p) = m_{\cA}(\mu_{J,h})$. Thus, the proof of Theorem \ref{th:basic} is complete once we establish the following statement. 

\begin{lemma}
If $p\in\cP_\cA(\O)$, then there exists a stationary state $\nu$ such that $p_t \to \nu$, as $t\to\infty$.
\end{lemma}
\begin{proof}
Fix $\mu=\mu_{J,h}$, for some $h\in\G_\cA$, and $p\in\cP_\cA(\O)$, and let $S_t(p)=p_t$ denote the solution of \eqref{non-linear rqs}.
By compactness, there exists a diverging sequence of times \((t_k)_{k \in \mathbb{N}}\) and \(\nu \in \mathcal{P}(\Omega)\) such that
\[
S_{t_k}(p) \to \nu \quad \text{as} \quad k \to \infty.
\]
By Lemma \ref{H-theorem}, it follows that
\begin{equation} \label{eq:abo}
\lim_{t \to \infty} H(S_t(p)\tc \mu) = \inf_{t \ge 0} \,H(S_t(p) \tc \mu) = H(\nu \tc \mu).
\end{equation}
Note that this holds for any $\mu=\mu_{J,h}$, $h\in\G_\cA$. If $\nu$ is stationary, then by Lemma \ref{lem:irre} and Lemma \ref{charact-of-stationarity} we know that $\nu=\mu_{J,h}$, for some $h\in\G_\cA$. Therefore, in \eqref{eq:abo} we can take $\mu=\nu$, which then implies that $H(S_t(p)\tc \nu) \to 0$, $t\to\infty$. In turn, e.g.\ by Pinsker's inequality, this implies the desired convergence $ S_t(p)\to \nu$. Thus, it suffices to prove that $\nu$ is stationary.

Assume, for contradiction, that \(\nu\) is not stationary. Then, by Lemma \ref{H-theorem} one has that $f=\nu/\mu$ satisfies $\cD_\mu(f)>0$.
Therefore, 
there exist \(u > 0\) and \(\delta > 0\) such that
\[
H(S_u(\nu) \tc \mu) \le H(\nu \tc \mu) - \d u .
\]
On the other hand, by the convergence $S_{t_k}(p) \to \nu$, and the continuity of $S_t(\cdot)$ and $H(\cdot\tc\mu)$, there exists $k\in\bbN$ such that $H(S_u(\nu) \tc \mu)\ge H(S_u(S_{t_k}(p)) \tc \mu) - \d u /2$.
Since $S_u(S_{t_k}(p))=S_{u+t_k}(p)$, this would then imply
\[
H(S_{u+t_k}(p) \tc \mu) \le H(\nu \tc \mu) - \d u/2 ,
\]
which contradicts \eqref{eq:abo}. This ends the proof. 
\end{proof}

Thus, we have proved Theorem \ref{th:basic}
assuming the validity of Lemma \ref{lem:irre}.
We now turn to the proof of Lemma \ref{lem:irre}.   
The strategy is based on ideas introduced in \cite{AC2024nonlinear} to prove convergence  for the discrete time version of the single site model in Example \ref{ex:ss}, where the irreducibility was deduced from the analysis of a fragmentation process and a coupon collector argument. 
However, the dynamics considered here differs from that in \cite{AC2024nonlinear} in two respects: first, it evolves in continuous time, and second, exchanges are allowed between any pair of sites. The passage to continuous time entails a natural modification of the original argument, which is based on a representation of the solution in terms of random binary trees. More substantially, allowing exchanges across arbitrary pairs of sites requires a new analysis, since fragmentation and coupon-collecting arguments alone are no longer sufficient; one must also account for the “mixing” within each class $A\in\cA$ in order to conclude.
%
The proof is broken down into several steps. We start with a graphical representation needed to take care of the continuous time nature of the process.

\subsection{Graphical representation}
%
We introduce some notation to handle binary rooted trees. 
The infinite rooted binary tree is represented as the set $\mathbb T:=\{\emptyset\} \cup_{k\ge 1} \{0,1\}^k$, where $\emptyset$ denotes the root, and an element $x\in\{0,1\}^k$ represents a node in the $k$-th generation. 
 For $x\in \bbT$ we let $|x|$ denote the length of the sequence $x$, interpreted as the depth of the node $x$.
 We equip $\mathbb T$ with the partial order $\prec$ defined by saying that $x \prec y$, if $|x| < |y|$ and $y=(x,z)$ for some $z\in \{0,1\}^{|y|-|x|}$, where  $(x,z)$ denotes the concatenation of the two sequences. A finite rooted binary tree is a finite subset $\gamma$ of $\mathbb T$ which satisfies the following:
\begin{enumerate}
		\item For any $x\in \gamma$ and any $y\in \mathbb{T}$, if $y \prec x$ then $y\in \gamma$.
		\item For any $x\in \gamma$ either $\{(x,0),(x,1)\}\subset \gamma$ or $\{(x,0),(x,1)\}\cap \gamma=\emptyset$.
	\end{enumerate}
	The set of maximal elements in $\gamma$ (for the order $\prec$) that is, its leaves, is denoted by
	\begin{equation}\label{leaf}
	L(\gamma):= \{ x\in \gamma \ : \ \nexists y\in \gamma, \  x \prec y\}.
	\end{equation}

Given a  finite rooted binary tree  $\g$, and $p\in \cP(\O)$, we define the distribution \( T_\g(p) \) recursively as follows. Assign to each node \( u \) in \( \g \) a probability distribution \( \nu_u \in \cP(\O) \), in such a way that each leaf node $u\in L(\g)$ is assigned the distribution \( p \), and each internal node is assigned the distribution resulting from the binary collision $\mu_1\circ\mu_2$, where the measures $\mu_1,\mu_2$ are assigned to its two children. The distribution \( T_\g(p) \) is defined as the distribution at the root of the tree obtained in this way. See Figure~\ref{fig:fig2} for an illustration.

More generally, given a collection of measures \( (p_1, p_2, \dots, p_L) \in \cP(\O)^L \) and a binary tree \( \g \) with \( L \) leaves \( L(\g) = \{x_1, \dots, x_L\} \), we define \( T_\g(p_1, \dots, p_L) \) by assigning \( p_i \) to the leaf \( x_i \), and then recursively applying the binary collision rule at internal nodes up to the root.

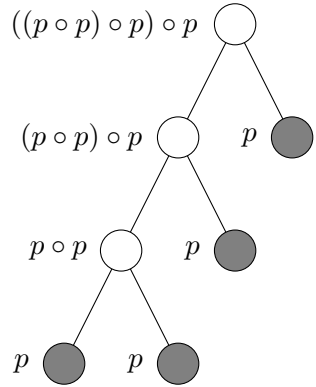
\begin{figure}[h]
\centering
\begin{tikzpicture}[scale=0.75]
    \draw  (1,-3) -- (4,3);
    \draw  (3,-3) -- (2,-1);
    \draw  (4,-1) -- (3,1);
    \draw  (5,1) -- (4,3);

    \node[shape=circle, draw=black, fill = white, scale = 1.5]  at (3,1) {};
    \node[shape=circle, draw=black, fill = white, scale = 1.5]  at (2,-1) {};
    \node[shape=circle, draw=black, fill = gray, scale = 1.5]  at (1,-3) {};
    \node[shape=circle, draw=black, fill = gray, scale = 1.5]  at (3,-3) {};
    \node[shape=circle, draw=black, fill = gray, scale = 1.5]  at (4,-1) {};
    \node[shape=circle, draw=black, fill = white, scale = 1.5]  at (4,3) {};
    \node[shape=circle, draw=black, fill = gray, scale = 1.5]  at (5,1) {};

    \node at (1/4,-3) {$p$};
    \node at (9/4,-3) {$p$};
    \node at (1-.1,-1) {$p \circ p$};
    \node at (13/4,-1) {$p$};
    \node at (17/4,1) {$ p$};
    \node at (6/4-.2,1) {$(p\circ p) \circ p$};
    \node at (8/4-.3,3) {$((p\circ p) \circ p)\circ p$};
\end{tikzpicture}
\caption{An example of a binary tree \( \g \) with four leaves, and the corresponding measure at the root: \( T_\g(p) = (((p \circ p) \circ p) \circ p) \).}
\label{fig:fig2}
\end{figure}

We now define the binary branching process, namely a stochastic process \( (\g_t)_{t \geq 0} \) taking values in the space of finite rooted binary trees. The process evolves as follows: start with \( \g_0 = \{\emptyset\} \), the single-node tree (the root). Each leaf of the tree is equipped with an exponential clock of mean one. When a clock associated with a leaf \( x \) rings at time \( t \), the leaf splits, and two children \( (x,0) \) and \( (x,1) \) are added to the tree, making \( x \) an internal node.
Equivalently, \( \g_t \) can be interpreted as a first-passage percolation process on the infinite binary tree \( \bbT \). Let \( (\cE_x)_{x \in \bbT} \) be a collection of i.i.d. exponential random variables with mean one. Define the passage time to a vertex \( x \in \bbT \) as
\begin{equation} \label{eq:exl}
\tau_x := \sum_{y \prec x} \cE_y,
\end{equation}
where the sum runs over all ancestors \( y \) of \( x \) (excluding \( x \) itself). Then, the random tree at time \( t \) is given by
\begin{equation} \label{defct}
\g_t := \left\{ x \in \bbT \ : \ \tau_x \leq t \right\}.
\end{equation}

We denote by \( {\bf P} \) the law of the random process \( (\g_t)_{t\ge 0} \), and by \( {\bf E} \) the associated expectation.
The following representation of the continuous time equation can be derived as is \cite[Lemma 3.1]{caputo2025cutoff} or, alternatively, as in \cite[Section 3.2]{CP2024nonlinear}.

\begin{lemma}\label{lem:rep_ct}
Let \( (\g_t) \) be the branching process defined above. Then, for any initial measure \( p \in \cP(\O) \), the solution \( p_t \) to the nonlinear equation \eqref{non-linear rqs}
is given by 
\begin{equation} \label{eq:solt}
p_t = {\bf E}\left[ T_{\g_t}(p) \right].
\end{equation}
\end{lemma}
%

\subsection{Reduction to discrete time and the case \(J=0\)}
The discrete time version of the equation \eqref{non-linear rqs} is the dynamical system $\Phi_k, k\in\bbN$, defined by 
\[\Phi_k(p) = \Phi_{k-1}(p)\circ \Phi_{k-1}(p)\,, \quad k\in\bbN,
\quad \Phi_0(p)=p\,.\] 
 Therefore, adopting the notation introduced above, the probability measure $\Phi_{k}(p)$ can be written as $T_{\G_k}(p)$, where $\G_k$ denotes the rooted regular binary tree with depth $k$; see Figure \ref{fig:fig1}
for an illustration.

\begin{figure}[h]
\centering
\begin{tikzpicture}[scale=0.8]
    \draw  (1,-1) -- (2,1);
    \draw  (3,-1) -- (2,1);
    \draw  (2,1) -- (4,3);
    \draw  (5,-1) -- (6,1);
    \draw  (7,-1) -- (6,1);
    \draw  (6,1) -- (4,3);
    
    \node[shape=circle, draw=black, fill = white, scale = 1.5]  at (4,3) {}; 
    \node[shape=circle, draw=black, fill = white, scale = 1.5]  at (2,1) {};
    \node[shape=circle, draw=black, fill = gray, scale = 1.5]  at (1,-1) {};
    \node[shape=circle, draw=black, fill = gray, scale = 1.5]  at (3,-1) {};
    \node[shape=circle, draw=black, fill = white, scale = 1.5]  at (6,1) {};
    \node[shape=circle, draw=black, fill = gray, scale = 1.5]  at (5,-1) {};
    \node[shape=circle, draw=black, fill = gray, scale = 1.5]  at (7,-1) {};

    \node at (1/4,-1) {$p$};
    \node at (9/4,-1) {$p$};
    \node at (17/4,-1) {$p$};
    \node at (25/4,-1) {$p$};
    \node at (1,1) {$p \circ p$};
    \node at (20/4,1) {$p \circ p$};
    \node at (7/4+.1,3) {$(p\circ p) \circ (p \circ p)$};
\end{tikzpicture}
\caption{$\Phi_{2}(p)=(p\circ p) \circ (p \circ p)$ is the distribution at the root of the binary tree $\G_2$ with depth $2$.}
\label{fig:fig1}
\end{figure}
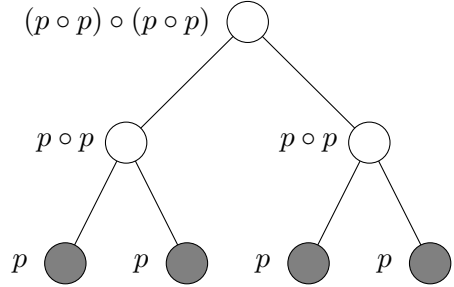

We are going to compare the continuous time process with the discrete time approximation, and then reduce to the simpler case \(J = 0\). 
As above, we extend the notation to include possibly distinct measures assigned to the leaves of the tree. Thus, we write \(\Phi_k(p_1,\dots,p_{2^k})=T_{\G_k}(p_1,\dots,p_{2^k})\), for the distribution at the root of the regular rooted tree $\G_k$ when the $2^k$ leaves are assigned the distributions $p_1,\dots,p_{2^k}\in\cP(\O)$. 
Similarly, we write \(\Phi^0_k(p_1,\dots,p_{2^k})\) for the corresponding quantity when collisions are computed with the $J=0$ kernel \(\mathcal{Q}_0^\mathcal{K}\) instead of \(\mathcal{Q}_J^\mathcal{K}\).
%
%

\begin{lemma}\label{Lemma reduction conv}
Fix a symmetric, stochastic $n\times n$ matrix $\cK$, and let $\cA$ denote its irreducible components. Fix \(m \in (-1,1)^{\mathcal{A}}\). 
Suppose there exist a positive constant \(\varepsilon := \varepsilon(m,\mathcal{K})\) and a positive integer \(u := u(m,\mathcal{K})\) such that
\begin{equation}\label{assumpJ=0}
    \Phi^0_u(p_1,\dots,p_{2^u})(\t) \geq \varepsilon, \quad \forall \t\in\O,
\end{equation}
for every choice of  \( (p_1, \ldots, p_{2^u})\) such that $m_\cA(p_i)=m$. 
Then,  there exist constants \(t_0>0\) and \(\varepsilon_0>0\) depending only on $m,\cK,J$,  such that for all $p\in\cP(\O)$ with $m_\cA(p)=m$, the solution \(p_t\) of \eqref{non-linear rqs} satisfies 
\begin{equation}\label{eq:desib}
p_t( \tau)\geq \varepsilon_0, \quad \forall\tau\in\Omega\,, \;\forall\,t\ge t_0\,.
\end{equation}
\end{lemma}

\begin{proof}
As a first step, we show that assumption \eqref{assumpJ=0} extends to the case of arbitrary interaction \(J\). Specifically, there exists a constant \(\varepsilon_1 := \varepsilon_1(J, \mathcal{K}, m) > 0\) 
such that
\begin{equation}\label{assump2J}
    \Phi_{u}(p_1,\dots,p_{2^u})(\t) \geq \varepsilon_1,
\end{equation}
for all $\t\in\O$, and for every choice of  \( (p_1, \ldots, p_{2^u})\) such that $m_\cA(p_i)=m$. Indeed, 
observe that 
\begin{equation}\label{eq:djq}
    \mathcal{Q}_J^\mathcal{K} = \delta_J \mathcal{Q}_0^\mathcal{K} + (1 - \delta_J) \overline{\mathcal{Q}},\qquad \delta_J := \tfrac12\,e^{-4\bar J}\,
\end{equation}
where \(\overline{\mathcal{Q}}\) is a probability kernel on \(\Omega \times \Omega\) that does not require explicit specification, and  $\bar J:=\max_i\sum_j|J_{i,j}|$. To prove \eqref{eq:djq} it suffices to use the definition \eqref{gxk2} and note that 
\begin{equation}
p_J(\ell,k\tc\si,\si')\geq \frac1{1+e^{4\bar J}}\,,
\end{equation}
for all \(\sigma, \sigma' \in \Omega\) and \(\ell, k,  \in [n]\). From the identity \eqref{eq:djq} it follows that, 
with probability $\delta_J$ the collision is performed according to the kernel $\mathcal{Q}_0^\mathcal{K}$. Since in the regular binary tree $T_{\G_u}$ there are $N=2^u-1$ collisions, one can write, 
for any $u\in\bbN$, and any choice of $p_1,\dots,p_{2^u}\in\cP(\O)$, 
\begin{equation}\label{eq:djqq}
   \Phi_u(p_1,\dots,p_{2^u})= \delta_J^N \Phi^0_u(p_1,\dots,p_{2^u}) + (1- \delta_J^N ) \Psi(u, p_1,\dots,p_{2^u})
   \,,
\end{equation}
where $\Psi(u, p_1,\dots,p_{2^u})\in\cP(\O)$ is some unspecified distribution.  
Therefore, the assumption \eqref{assumpJ=0}  implies \eqref{assump2J} with the same $u$ and with $\e_1= \e\delta_J^N$. 

Now, fix \(p \in \mathcal{P}(\Omega)\) such that \(m_{\mathcal{A}}(p) = m\). Using the representation formula for \(p_t\) from Lemma~\ref{lem:rep_ct}, and considering the regular tree \({\G_u}\), one has the lower bound
\begin{equation}
    p_t(\t) = {\bf E}[T_{\gamma_t}(p)(\t)] \geq {\bf E}[T_{\gamma_t}(p)(\t) ;\,{\G_u} \subset \gamma_t].
\end{equation}

On the event \({\G_u} \subset \gamma_t\), there exists a (random) vector \( (\mu_1, \ldots, \mu_{2^{u}})\), where each \(\mu_i\in\cP(\O)\) is the result of recursive collisions along subtrees rooted at the leaves of 
\({\G_u}\). Due to the conservation law in Lemma~\ref{conserved magn}, each distribution \(\mu_i \) satisfies $m_\cA(\mu_i)=m_\cA(p)=m$, and therefore the inequality \eqref{assump2J} applies. It follows that, whenever $ {\G_u} \subset \gamma_t$, 
\begin{equation}
    T_{\gamma_t}(p)(\t) = \Phi_u(\mu_1,\dots,\mu_{2^u})(\t)\ge  \varepsilon_1.
\end{equation}
In conclusion, for all $\t\in\O$,
\begin{equation}
    p_t(\t) \geq \varepsilon_1 {\bf P}({\G_u}  \subset \gamma_t).
\end{equation}
To estimate the latter probability, note that ${\G_u}  \subset \gamma_t$, iff for each leaf \(x \in L(\G_u)\), its associated arrival time satisfies \(\tau_x \leq t\), where \(\tau_x\) is distributed as a Gamma(\(u, 1\)) random variable; see \eqref{eq:exl}. By a union bound, and using a standard tail estimate, for $t>u$, 
\begin{equation}
    {\bf P}({\G_u}  \subset \gamma_t) \geq 1-2^{u} e^{-t} (e t/u)^u \,.
\end{equation}
Choosing \(t_0 = t_0(u) > 0\) such that the above probability is at least $1/2$, one concludes the desired bound \eqref{eq:desib} with $\e_0=\e_1/2$.
\end{proof}

\subsection{Proof of Lemma \ref{lem:irre}}
According to Lemma \ref{Lemma reduction conv}, it suffices to establish that the assumption \eqref{assumpJ=0} is satisfied. Note that in this case one actually obtains a slightly stronger statement, since the bound \eqref{eq:irreb} is uniform over all $p\in\cP(\O)$ with the given magnetization vector $m_\cA(p)=m\in(-1,1)^\cA$.

We now focus on establishing the validity of the assumption \eqref{assumpJ=0}. 
The set of stationary measures consists of non-homogeneous products of Bernoulli measures $\pi=\pi^\a$, where the parameters $\a=(\a_i)\in(0,1)^n$ are constant within each component of the partition \(\mathcal{A}\),
that is \[
\pi(\sigma) = \textstyle{\prod_{i=1}^n} \pi_i(\sigma_i),
\]
with \(\pi_i(\si) = {\rm Be}(\a_i)$, and \(\a_i = \a_j\) for all \(i, j \in A\), for all \(A \in \mathcal{A}\); see Lemma~\ref{charact-of-stationarity}.
Moreover, for any $m\in(-1,1)^\cA$ there is a unique vector $\a$ with \(\a_i = \a_j\) for all \(i, j \in A\), for all \(A \in \mathcal{A}\), such that the associated measure $\pi=\pi^\a$ satisfies $m_\cA(\pi)=m$; see Lemma \ref{lem:uni}.
\begin{lemma}\label{J=0 irredu}
Fix $m\in(-1,1)^\cA$, and let \(\pi\) denote the associated product of Bernoulli measures as above with \(m_\cA(\pi)=m\). Then, for every \(\varepsilon > 0\), there exists  \(u := u(\varepsilon, \mathcal{K}, m)\) such that
\begin{equation}\label{eq:toprotv}
    \|\Phi_u^0(p_1,\dots,p_{2^u}) - \pi\|_{\mathrm{TV}} \leq \varepsilon,
\end{equation}
for all choices of \(p_1, \ldots, p_{2^u}\in\cP(\O)\) with  \(m_\cA(p_i) = m\).
\end{lemma}
Before proving the lemma we show that  the assumption \eqref{assumpJ=0} is satisfied if the lemma holds. Indeed, 
since $m\in(-1,1)^\cA$ one has $\pi_* := \min_{\sigma \in \Omega} \pi(\sigma) > 0$. Thus, by 
Lemma~\ref{J=0 irredu} with \(\varepsilon= \pi_*/2\), and by definition of the total variation distance, 
\begin{equation}\label{assumpJ=00}
    \Phi^0_u(p_1,\dots,p_{2^u})(\t) \geq \pi(\t) - \frac{\pi_*}{2} \geq \e, \quad \forall \t\in\O.
\end{equation}
This ends the proof of Lemma \ref{lem:irre} assuming the validity of Lemma \ref{J=0 irredu}.

\subsection{Marked Partition Process}
We prove Lemma \ref{J=0 irredu} by deriving an explicit representation of the distribution $\Phi^0_u(p_1,\dots,p_{2^u})$ via a suitable discrete time fragmentation process, augmented with randomly walking labels, which we call the {\em Marked Partition Process} (MPP). Let $X=\{0\}\cup[n]$ denote the space of marks. At each time $u\in\bbN$, the marked partition process takes a value $\cC^u$ of the form 
$\cC^u=(C_1^u,\dots,C_{2^u}^u)$, with each fragment $C_i^u$  given by $C_i^u=(A^u_i,x^u_i)$, where $A^u_i\subset [n]$ is called the {\em spatial part}, and $x^u_i\in X$ is called the {\em mark} of the fragment $C_i^u$. The spatial parts are such that $\{A^u_i,\,i=1,\dots,2^u\}$ form a partition of $[n]$, for each $u\in\bbN$, that is $A^u_i\cap A^u_j=\emptyset $ for $i\neq j$ and $\cup_i A^u_i=[n]$ (the sets $A^u_i$ can be empty). 
To define the evolution of the process we introduce some more notation.  For any $x\in[n]$ we let \(X^x(u)\), for \(u \in \mathbb{N}\) denote the Markov chain with initial state \(x\) and transition matrix \(\mathcal{K}\). Similarly, we let \(Y^x(u)\), for \(u \in \mathbb{N}\), be the Markov chain with initial state $x \in [n]$, and transition matrix given by the lazy kernel
\[
\bar{\mathcal{K}} = \frac{1}{n}\mathcal{K} + \frac{n-1}{n} I,
\]
where \(I\) is the identity $n\times n$ matrix. With this notation, the MPP $\cC^0,\cC^1,\dots$ is defined as follows. 
At time $0$, we start with the single fragment $\cC^0=(C_1^0)$, where $C_1^0=([n],0)$. Recursively,  the process is defined by the following splitting and marking rules.
If $\cC^u=(C_1^u,\dots,C_{2^u}^u)$, 
then, independently for every $i=1,\dots,2^u$, sample a uniform random variable $U\in[n]$ and a uniform random variable $B\in\{1,2,3,4\}$, and set $
\cC^{u+1}=(C_1^{u+1},\dots,C_{2^{u+1}}^{u+1})
$, according to the rules
\begin{enumerate}
\item if $B=1$, then $(C_{2i-1}^{u+1},C_{2i}^{u+1})=(C_i^{u},(\emptyset,0))$, 
\item if $B=2$, then $(C_{2i-1}^{u+1},C_{2i}^{u+1})=((\emptyset,0),C_{i}^{u})$,
\item if $B=3$, $C_i^{u}=(A,0)$ and $U\in A$, then $(C_{2i-1}^{u+1},C_{2i}^{u+1})=((A\setminus\{U\},0),(\{U\},X^U(1))$,
\item if $B=3$, $C_i^{u}=(A,0)$ and $U\notin A$, then $(C_{2i-1}^{u+1},C_{2i}^{u+1})=((A,0),(\emptyset,0))$,
\item if $B=3$, $C_i^{u}=(\{i\},x)$, $x\in[n]$, then $(C_{2i-1}^{u+1},C_{2i}^{u+1})=((\{i\},Y^x(1)),(\emptyset,0))$,
\item if $B=4$, repeat the $B=3$ rules and then invert the order of $(C_{2i-1}^{u+1},C_{2i}^{u+1})$.

\end{enumerate}
We remark that, since we start with $\cC^0=(C_1^0)$, where $C_1^0=([n],0)$, the above rules are such that a nonzero mark can only be assigned to a spatial part consisting of a single site.  Thus, the above rules 
yield a well defined evolution. Note that, by construction, the spatial parts are such that $(A^u_i,\,i=1,\dots,2^u)$ forms a partition of $[n]$, at any time $u$.  We remark that the partition $(A^u_i)$ is monotone in the sense that  it can only become finer with time, until it consists only of empty sets or single sites with nonzero marks. One can define the {\em fragmentation time}
\begin{equation}\label{eq:fragt}
H = \min \left\{ u \in \mathbb{N} : \text{all nonempty sets have nonzero marks in } \cC^u \right\}.
\end{equation}
For each time $u$, there is at most one spatial part $A_i^u$ which can have cardinality larger than $1$, and, with probability $1/2$ it loses one element every time the uniform random variable  $U\in[n]$ belongs to it. Therefore, the erosion of the largest spatial part can be described by a simple coupon collecting process, where at each step, with probability $1/2$, a coupon is chosen uniformly at random, while with probability $1/2$ no coupon is collected. The fragmentation time $H$ coincides with the first time where $n$ distinct coupons have been collected, and thus a union bound gives
\begin{equation} \label{coupestimateH}
    \mathbb{P}(H \ge u) \le n\,\left(1-\tfrac1{2n}\right)^u \le n\,e^{-u/2n}\,.
\end{equation}

We note that there is a natural representation of the Marked Partition Process in terms of a regular binary tree. Namely, we can assign to each node of the rooted regular binary tree with depth $u\in\bbN$, $\G_u=\{\emptyset\}\cup_{k=1}^u\{0,1\}^k$, a marked set in such a way that the root is assigned the set $C_1^0=([n],0)$, and if the nodes with depth $k$ are assigned, say from left to right,  the sets $(C_1^k,\dots,C_{2^{k}}^k)$, then the nodes at the next generation are assigned, respectively, from left to right, the sets $(C_1^{k+1},\dots,C_{2^{k+1}}^{k+1})$, with the pair $(C_{2i-1}^{k+1},C_{2i}^{k+1})$, being assigned to the two children of the node corresponding to $C_i^{k}$. In this way, starting from the root, the marks along any path is always $0$ until the first time where it takes some value $x\in[n]$, after which it evolves, along the unique  branch with nonempty fragments, with marks given the lazy Markov chain $ \bar{\mathcal{K}}$. In particular, any marked set of the form $(\{i\},x)$, $x\neq 0$, must be such that $i$ and $x$ belong to the same irreducible component $A\in\cA$. 

We now describe how the above defined process encodes the measure $\Phi^0_u(p_1,\dots,p_{2^u})$.  
 For any $u\in\bbN$, and any realization of $\cC^u$, for any vector $\vec p(u)=(p_1,\dots,p_{2^u})$, 
define the measure $\Psi(\vec p(u),\cC^u)\in\cP(\O)$ by the following rule. Given $\cC^u=(C_1^u,\dots,C_{2^u}^u)$, there is at most one marked set $C_{i}^u$ of the form $(A,0)$, with $\emptyset\neq A\subset [n]$ and mark $0$, and all other marked sets are either empty or of the form $(\{j\},x_j)$ for some $j\in [n]$ and some nonzero mark $x_j\in[n]$. For any $\t\in\O$, $p\in\cP(\O)$, $i=1,\dots,2^u$, we define 
 \begin{equation} \label{eq:prepres}
\psi(p,C_i^u)(\t) = \begin{cases}
p_{A}(\t_{A}) & \text{ if } C_{i}^u=(A,0)\\
p_{x}(\t_{j}) & \text{ if } C_{i}^u=(\{j\},x)
\end{cases} 
\end{equation}
In this formula, we use the notation $p_{A}$ for the marginal at $A$ of the measure $p$ and $p_{x}$ for the marginal at $\{x\}$ of the measure $p$, and as usual $\t_A$ denotes the spin configuration $(\t_x, \,x\in A)$, with $\t_x$ denoting the spin at $x$. Moreover, we adopt the convention that, when $A=\emptyset$, then we set $p_{A}(\t_{A})=1$.
Finally, we set
 \begin{equation} \label{eq:prepres1}
\Psi(\vec p(u),\cC^u)(\t)=\prod_{i=1}^{2^u}\psi(p_i,C_i^u)(\t)\,.
\end{equation}
Note that $\Psi(\vec p(u),\cC^u)\in\cP(\O)$, for any $u$, any choice of the distribution vector $\vec p(u)=(p_1,\dots,p_{2^u})$, and for any realization of the MPP $(\cC^u, u=0,1,\dots)$.  
In words,  $\Psi(\vec p(u),\cC^u)$ is obtained by taking products of the marginals over the largest spatial part $(A,0)$ in the marked partition $\cC^u$, and the marginals over the single sites identified by the nonzero marks, with the caveat that the marginal of $\Psi(\vec p(u),\cC^u)$ on the single sites appearing with a nonzero mark is given by the single site marginal identified by the nonzero mark. 
\begin{lemma}\label{lem:repres}
For any $u\in\bbN$, $\t\in\O$, any $p_1,\dots,p_{2^u}\in\cP(\O)$, 
 \begin{equation} \label{eq:prepres2}
\Phi^0_u(p_1,\dots,p_{2^u})=\bbE\left[\Psi(\vec p(u),\cC^u)(\t)\right]\,,
\end{equation}
where the expectation is over the MPP. 
\end{lemma}
\begin{proof}
We use induction over $u$. Note that it is true at $u=0$, since $\cC^u$ is given by the single marked set $C_1^0=([n],0)$ and $\Phi^0_u(p)=p$. Let us check that for all $p,q\in\cP(\O)$, 
 \begin{equation} \label{eq:prepres21}
\Phi^0_1(p,q)(\t)=p\circ q\,(\t)=\bbE[\psi(p,C_1^1)(\t)\psi(q,C_2^1)(\t)]=\bbE\left[\Psi(p,q,\cC^1)(\t)\right]\,,
\end{equation}
which is the identity \eqref{eq:prepres2} for $u=1$. 
To see this, note that
by definition of the collision kernel \(\mathcal{Q}_0^\mathcal{A}\), in the case $J=0$ one has
\begin{align}\label{defT01}
   & (p \,\circ\, q)(\tau)  = 
\frac14 \,p(\t)+\frac14\, q(\t)  \;+\nonumber \\
& \;\;\;  +\frac1{4n}  \sum_{\ell, k \in [n]} \mathcal{K}(\ell, k) \left[ p_{[n] \setminus \{\ell\}}(\tau_{[n] \setminus \{\ell\}}) \, q_k(\tau_\ell) + p_k(\tau_\ell) \, q_{[n] \setminus \{\ell\}}(\tau_{[n] \setminus \{\ell\}})\right],
\end{align}
where $p_A$ denotes the marginal of $p$ over a subset $A\subset[n]$, and we use $p_k$ for the marginal of $p$ over a single site $\{k\}$.
Points (1) and (2) in the definition of the MPP take care of the terms $\frac14\,p$ and $\frac14\,q$ in \eqref{defT01}.
Moreover, points (3)-(4)-(5) provide the term \[
\frac14\,p_{[n] \setminus \{U\}}(\tau_{[n] \setminus \{U\}}) \, q_{X^U(1)}(\tau_U)\] 
in \eqref{defT01}. Finally, point (6) provides the term \[\frac14\, p_{X^U(1)}(\tau_U)\,q_{[n] \setminus \{U\}}(\tau_{[n] \setminus \{U\}}) \] in \eqref{defT01}.
Taking the expectation over $U,X^U(1)$ yields \eqref{eq:prepres21}.
  
  Reasoning in the same way, one obtains, for any realization of $C^u_i$,  for all $\t$, 
  \begin{align}\label{defT01a}
   \psi(p\circ q,C^u_i)\,(\t)=
   \bbE\left[\psi(p,C^{u+1}_{2i-1})(\t)\psi(q,C^{u+1}_{2i})(\t)\tc C^u_i\right],
   \end{align}
   where $\bbE[\cdot\tc C^u_i]$ denotes the expectation over one step of the MPP  $C^u_i \mapsto (\cC^{u+1}_{2i-1},\cC^{u+1}_{2i})$, conditionally on $C^u_i$.

Assuming the validity of \eqref{eq:prepres2} at $u$, the formula for $u+1$ can be derived by using \eqref{eq:prepres21} applied to the collisions $ p_{2i-1}\circ p_{2i}$, for every $i=1,\dots,2^u$. Indeed, setting $\vec q(u)=(q_1,\dots,q_{2^u})$ with $q_i :=p_{2i-1}\circ p_{2i}$,  by repeated application of \eqref{defT01a}, for all $\t\in\O$, 
\begin{align} \label{eq:prepres3}
&\bbE\left[\Psi(\vec p(u+1),\cC^{u+1})(\t)\tc\cC^s,\,s\le u\right] 
\nonumber \\& \; =\prod_{i=1}^{2^{u}}\bbE\left[\psi(p_{2i-1},C_{2i-1}^{u+1})(\t)\psi(p_{2i},C_{2i}^{u+1})(\t)\tc\cC^s,\,s\le u\right]=\Psi(\vec q(u),\cC^u)(\t)\,.
\end{align}
Taking the expectation, and using the inductive assumption, one finds 
\[
\bbE\left[\Psi(\vec p(u+1),\cC^{u+1})\right]=\Phi^0_u(q_1,\dots,q_{2^u})\,.
\]
Since $\Phi^0_u(q_1,\dots,q_{2^u})=\Phi^0_{u+1}(p_1,\dots,p_{2^{u+1}})$, this proves the desired identity \eqref{eq:prepres2} for all $u$ and all $\vec p(u)$. 
\end{proof}
\subsection{Proof of Lemma \ref{J=0 irredu}}
Once the representation in terms of the MPP is established we turn to the proof of Lemma \ref{J=0 irredu}. To prove \eqref{eq:toprotv}, we observe that for any $u_0<u$, using Lemma \ref{lem:repres}, Jensen's inequality and the fact that $\{H\le u_0\}$ is measurable with respect to $\cC^{\le u_0}:=(\cC^{s}, s\le u_0)$, one has
\begin{align}\label{eq:toprotv2}
 &   \|\Phi_u^0(p_1,\dots,p_{2^u}) - \pi\|_{\mathrm{TV}} \leq \bbP(H>u_0) + \nonumber\\ &
\qquad +\,    \bbE\left[\left\|\bbE\left[\Psi(p_1,\dots,p_{2^u},\cC^u)\tc \cC^{\le u_0}\right] - \pi\right\|_{\mathrm{TV}};H\le u_0\right]\,.
\end{align}
We take $u_0=u_0(\e,n)$ such that $n\,e^{-u_0/2n}\le \e/2$, so that by \eqref{coupestimateH} the first term above is at most $\e/2$.
To estimate the second term, the key observation is that conditionally on $\cC^{\le u_0}$,  on the event $H\le u_0$, one has the non-homogeneous product of Bernoulli measures
\begin{align}\label{eq:toprotv3}
 \bbE\left[\Psi(p_1,\dots,p_{2^u},\cC^u)\tc \cC^{\le u_0}\right] =\textstyle{\bigotimes_{j=1}^n}{\rm Be}(\xi_j)\,,
\end{align}
where, for each $j\in[n]$, the parameter $\xi_j=\xi_j(p_1,\dots,p_{2^u},\cC^{\le u_0})$ is obtained as follows. Let $C_{i_*(j)}^u=(\{j\},x_*(j))$ denote the unique marked set  with spatial part given by $\{j\}$, and let $W_j=p_{i_*(j)}(\si_{x_*(j)}=+1)$ denote the probability of the event $\si_{x_*(j)}=+1$, when $\si$ is distributed according to $p_{i_*(j)}$. Then 
\[
\xi_j=\bbE\left[W_j\tc\cC^{\le u_0}\right]\,.
\]
Since the total variation distance between two Bernoulli variables is the distance between the parameters, one has
\[
\left\|\bbE\left[\Psi(p_1,\dots,p_{2^u},\cC^u)\tc \cC^{\le u_0}\right] - \pi
\right\|_{\mathrm{TV}}
\le n\,\max_j |\xi_j - \a_j|\,.
\]
An important feature is that the random site $x_*(j)$, conditionally on $\cC^{\le u_0}$,  on the event $H\le u_0$, has the distribution of $Z^j(u-u_*(j))$, where $u_*(j)\le u_0$ denotes the first time the site $j$ receives a nonzero mark in the MPP, and  $Z^j(k)$ stands for the distribution after $k$ steps of the Markov chain with transition $\widehat\cK:=\frac12 I + \frac12 \bar\cK$, with initial state $j$. On the other hand we note that the parameter $\a_j$, by our assumption $m_\cA(p_i)=m$ for all $p_i$, satisfies $\a_j=\bbE[p_{i}(\si_{y_*(j)}=+1)]$, for any $i=1,\dots,2^u$, where $y_*(j)$ is distributed uniformly in the irreducible component $A(j)\in\cA$ containing $j$. Call $u(j)$ the uniform distribution over $A(j)$. As a consequence, we may estimate the difference $|\xi_j - \a_j|$ by using the optimal coupling between the chain with transition $\widehat \cK$ and the uniform distribution over $A(j)$.  Since $u_*(j)\le u_0$, we can estimate
\begin{align}\label{eq:toprotv4}
 \max_j |\xi_j - \a_j|
 \le \max_j \|\widehat\cK^{u-u_0}(j,\cdot) - u(j)\|_{\rm TV}\,.
\end{align}
By elementary mixing time theory, for any $\e>0$ there exists an integer $u>u_0$ such that the above is at most $\e/2$. Together with \eqref{eq:toprotv2} this concludes the proof.

\section{Relative entropy decay for the particle system}\label{sec:MLSI}
This section is concerned with the proof of Theorem \ref{th:MLSI}.  We start with a version of that theorem that covers the case of the mean field exchange dynamics in Example \ref{ex:mf}, that is for the kernel $\cK(\ell,n)\equiv1/n$, and then turn to the proof for the general multi-component case in Definition \ref{ex:mcmf}. A key intermediate step in these proofs will be a modified log-Sobolev inequality for a version of the so-called Down-Up walk for the canonical Ising measure.  

\subsection{Mean field exchange dynamics}
Here we state the main result for the kernel $\cK(\ell,n)\equiv1/n$. In order to prepare for later applications to multi-component systems we need to extend the setup from Section \ref{subsec:kac} to include arbitrary external fields. Thus, for $h\in\bbR^{Nn}$, we write $h=(h(i),i\in[N])$, and $h(i)\in\bbR^n$ for all $i\in[N]$, and define
\begin{equation}\label{eq:pbm2}
\mu_{N,J,h,\r}:=\otimes_{i=1}^N\mu_{J,h(i)}(\cdot\tc\O_{N,\r}), 
\end{equation} 
the measure obtained by taking products of $\mu_{J,h(i)}$ and then restricting to  $\O_{N,\r}$.
Recall the notation $\si^{i,\ell;j,k}$ for the configuration $\si\in\O^N$ obtained by exchanging the spins $\si_\ell(i)$ and $\si_k(j)$; see \eqref{eq:unde}.  We use the notation 
 \begin{equation}\label{gxk33}
    r_{J,h,i,j}(\ell,k\tc\sigma(i),\sigma(j))=
 \frac{ \mu_{N,J,h,\r}(\si^{i,\ell;j,k})}{\mu_{N,J,h,\r}(\si)+\mu_{N,J,h,\r}(\si^{i,\ell;j,k})}\,.
\end{equation}
Note that, because of the product structure, this quantity depends on $\si $ only through $\si(i)$ and $\si(j)$. Moreover, it coincides with \eqref{hatp} when $h\equiv 0$. We define the Dirichlet form 
\begin{equation}\label{bar-DF}
\bar{\mathcal{E}}_{N, J,h,\r}(F, G) := \frac{1}{2Nn} \sum_{i,j=1}^N \sum_{\ell,k =1}^n \mu_{N,J,h,\r}\left[  r_{J,h,i,j}(\ell,k\tc\sigma(i),\sigma(j)) \nabla^{i,\ell;j,k}F\nabla^{i,\ell;j,k}G \right],
\end{equation}
for functions \(F, G : \O_{N,\r} \to \mathbb{R}\). As usual, we let 
\[
\operatorname{Ent}_{\mu_{N,J,h,\r}}(F)=\mu_{N,J,h,\r}\left[F\log(F/\mu_{N,J,h,\r}[F])\right]
\] denote the entropy of $F\ge 0$ with respect to $\mu_{N,J,h,\r}$.
\begin{theorem}\label{th:MLSI-MF}
Let $h\in\bbR^{Nn}$, let \(J\in\bbR^{n\times n}\) be a nonnegative definite matrix with largest eigenvalue \(\lambda(J) < 1/2\), and let  
\(\overline{J}=\max_{i\in [n]} \sum_{j\in [n]}|J_{i,j}|\), \(\overline{h}=\max_{\ell\in [n], i\in[N]}|h_\ell(i)|\). 
For any $N\ge 2$, and $\r\in \cD_{\cA,N}$, 
\begin{equation}\label{eq:MLSI-MF1}
\bar{\mathcal{E}}_{N, J,h,\r}(F, \log F) \ge \a\operatorname{Ent}_{\mu_{N,J,h,\r}}(F),
\end{equation}
for all $F\ge 0$, with constant  $\a$ such that  
  \begin{equation}\label{eq:MLSI-MF2}
        \alpha \geq \tfrac14(1 - 2\lambda(J))\,e^{-8(\overline{J}+\overline{h})}.
   \end{equation}
\end{theorem}
We note that Theorem \ref{th:MLSI-MF} implies Theorem \ref{th:MLSI} in the special case $\cK(\ell,n)\equiv1/n$, since $(1-2\lambda(J))\le 1$ and for this choice of kernel, taking $h\equiv 0$,  one has \[\bar{\mathcal{E}}_{N, J,0,\r}(F, G)=n\mathcal{E}_{N, J,\r}(F, G)\,,\]
where $ \mathcal{E}_{N, J,\r}(F, G)$ is given in \eqref{Dirichlet form} with $\cK(\ell,n)\equiv1/n$.
The proof of Theorem \ref{th:MLSI-MF} will be derived as a consequence of the following more general estimate for mean field exchange systems.

\subsection{Down-Up Walk for Conservative Particle Systems}
We introduce some new notation. Consider a generic Ising spin system, with $L$ spins, for some \(L \in \mathbb{N}\), with interaction given by an $L\times L$ symmetric matrix $\L$ and with external field $w\in\bbR^L$, that is
 \begin{equation}\label{L-Ising-Measure}
\mu_{\Lambda,w}(\eta) =\frac1{Z_{\L,w}} \exp\left( \frac{1}{2} \langle \eta, \Lambda \eta \rangle + \langle \eta, w \rangle \right),\qquad \eta\in\{-1,1\}^L\,,
\end{equation}
where $ \langle \eta, \Lambda \eta \rangle=\sum_{i,j=1}^L\eta_i\L_{i,j}\eta_j$ and $\langle \eta, w \rangle=\sum_{i=1}^Lw_i\eta_i$. Given an integer $M\in\{-L,\dots,L\}$, consider the state space
\begin{equation}\label{OLM}
\Omega_{L,M} = \left\{ \eta\in\{-1,1\}^L, \  \textstyle{\sum_{i=1}^L} \eta_i = M\right\}.
\end{equation}
We will assume that $L+M$ is even, so that $\O_{L,M}$ is not empty. 
We view a configuration $\eta\in\O_{L,M}$ as composed of $K$ points in the set $[L]=\{1,\dots,L\}$, corresponding to the sites $i$  where $\eta_i=+1$. Thus, $\eta$ is a configuration of $K=(L+M)/2$ elements from $[L]$, which we view as sites occupied by ``balls". The remaining sites are such that $\eta_i=-1$ and will be seen as empty or ``holes". The Down-Up Walk on $\Omega_{L,M}$  for the Ising system $\mu_{\Lambda,w}$ is the continuous-time Markov  chain where a ball at site $i$, with rate 1 is moved to a hole at $j$, chosen among all possible vacant locations (including going back to $i$ itself) with probability proportional to $\mu(\eta^{i,j})$, where \(\eta \to \eta^{i,j}\) is the configuration obtained by swapping the spins values at sites \(i\) and \(j\). Using the notation $\nabla^{i,j}F(\eta)=F(\eta^{i,j}) - F(\eta)$, the infinitesimal generator of the process is 
\begin{equation}\label{GenDU}
\mathcal{L}^{\rm du}F(\eta) = \sum_{i,j=1}^L c^{\rm du}(\eta,\eta^{i,j}) \nabla^{i,j}F(\eta),
\end{equation}
with $F:\Omega_{L,M}\mapsto \bbR$, and \(c^{\rm du}(\eta,\eta^{i,j})\) is the transition rate 
\begin{equation}\label{RatesDU}
c^{\rm du}(\eta,\eta^{i,j}) = \mathbf{1}(\eta_i = +1,\eta_j = -1)
\frac{\mu_{\L,w}(\eta^{i,j})}{\sum_{k \in U_i(\eta)} \mu_{\L,w}(\eta^{ik})},
\end{equation}
where \(U_i(\eta) = \{j \in [L] : \eta_j = -1\} \cup \{i\}\) denotes the set of available sites. This generator defines a conservative dynamics, reversible with respect to \(\mu_{\L,w}\).
It is not hard to check that, for any $M$ such that $L+M$ is even, the process is irreducible in the state space $\O_{L,M}$, and the restricted measure 
\begin{equation}\label{RestrDU}
\nu_{L,M}= \mu_{\L,w}\left(\cdot\tc\O_{L,M}\right)
\end{equation}
is reversible, and thus the unique invariant measure. Let \(\mathcal{E}^{\rm du}(F,G)\) denote the associated Dirichlet form,
\begin{equation}\label{DirDU}
\mathcal{E}^{\rm du}(F,G) =  \frac{1}{2} \sum_{i,j=1}^L \nu_{L,M}
\left[ c^{\rm du}(\eta,\eta^{i,j})\nabla^{i,j}F(\eta)\nabla^{i,j}G(\eta) \right]\,, 
\end{equation}
for \(F,G : \Omega_{L,M} \to \mathbb{R}\).
Let also $\ent_{\nu_{L,M}}(F)$ denote the entropy with respect to the canonical Ising measure $\nu_{L,M}$.
  The key to prove Theorem \ref{th:MLSI-MF} will be the following MLSI for the Down-Up walk. 

\begin{theorem}[MLSI for the Down-Up walk]\label{MLSI DOWN UP}
Suppose that \(\Lambda\) is a nonnegative definite matrix with largest eigenvalue \(\lambda(\Lambda) < 1/2\). Then 
for all values of $L,M$, with $L+M$ even, for all $w\in\bbR^L$,
\[
\mathcal{E}^{\rm du}(F, \log F)\ge (1 - 2\lambda(\Lambda))\, \mathrm{Ent}_{\nu_{L,M}}(F) \,,
\]
for every \(F : \Omega_{L,M} \to \bbR_+\).
\end{theorem}
It is crucial that the estimate is uniform over the parameters $L,M$ and the external fields $w\in\bbR^L$. The above theorem can be proven using the arguments in Theorem 2.2 and Lemma 2.7 in \cite{bauerschmidt2024kawasaki}. Moreover, following \cite{bauerschmidt2024kawasaki} it can be shown that if the largest eigenvalue of $\L$ is associated to a constant eigenvector, then one can impose the condition \(\lambda(\Lambda) < 1/2\) on the second largest eigenvalue. This is natural, in light of the global constraint $\sum_i\eta_i=M$. Indeed, if e.g.\ $\L_{i,j}\equiv c$, then the canonical measure $\nu_{L,M}$ becomes uniform and one might as well take $\L=0$. However, for the applications to our setting, the removal of the eigenvalue associated to constant eigenfunction will not be relevant. Below, we prove Theorem \ref{MLSI DOWN UP} with a slightly different argument, which will allow us to introduce the relevant tools to prove the more general statements for multi-component mean field exchange models, which are not covered by the results in  \cite{bauerschmidt2024kawasaki}. Our proof will follow the factorization technique developed in \cite{caputo2025factorizations}, which in turn was based on the stochastic localization approach from \cite{chen2022localization}. 
The proof of Theorem \ref{MLSI DOWN UP} will be given in the following subsections. Here we observe that it implies Theorem \ref{th:MLSI-MF}.

\begin{proof}[Proof of Theorem \ref{th:MLSI-MF} assuming Theorem \ref{MLSI DOWN UP}]
Let \(L = Nn\), where \(N\) is the number of particles and \(n\) the number of spins per particle. We use the correspondence
\begin{equation}\label{eq:corres}
\eta \in \{-1,1\}^{Nn} \longrightarrow \sigma = (\sigma(i), i = 1, \dots, N), \quad \sigma(i) = (\eta_u, \,u = n(i-1)+1, \dots, ni),
\end{equation}
that is the spins in the block of sites \(\{(i-1)n+1,\dots,ni\}\) are associated with the $i$-th particle \(\si(i)=(\si_1(i),\dots,\si_n(i))\). 
The $L\times L$ matrix $\L$ and the $L$-vector $w$ are taken such that  
\begin{equation}\label{def JN}
    \langle \eta, \L \eta \rangle = \textstyle{\sum_{i=1}^N} \langle \si(i), J \si(i) \rangle, \qquad  \langle \eta, w \rangle = \textstyle{\sum_{i=1}^N}\langle \si(i), h (i)\rangle
\end{equation}
where the scalar products on the left hand side are over $\bbR^L$ while each of the scalar products on the right hand side is over $\bbR^n$. In this way, taking $M=(2\r-1)Nn$, the canonical measure $\nu_{L,M}$ in \eqref{RestrDU} coincides with the probability measure $\mu_{N,J,h,\r}$. Since $\L$ consists of $N$ blocks given by the $n\times n$ matrix $J$, one has $\l(\L)=\l(J)$.  
Therefore, to prove \eqref{eq:MLSI-MF1} it remains to establish the comparison
\begin{equation}\label{eq:MLSI-MF11}
\bar{\mathcal{E}}_{N, J,h,\r}(F, \log F) \ge \tfrac14\,e^{-8(\overline{J}+\overline{h})}\mathcal{E}^{\rm du}(F, \log F),
\end{equation}
for all nonnegative functions $F$. To prove this, observe that, if $M\le 0$, that is $\r\le1/2$, then $|U_i(\eta)|\ge L/2=Nn/2$, for all $\eta\in\O_{L,M}$, and therefore
\begin{align}\label{RatesDU2}
c^{\rm du}(\eta,\eta^{i,j}) &\le \frac4{Nn}\,e^{8(\overline{J}+\overline{h})}\frac{\mu_{\L,w}(\eta^{i,j})}{\mu_{\L,w}(\eta)+\mu_{\L,w}(\eta^{i,j})}\nonumber\\
&
=  \frac4{Nn}\,e^{8(\overline{J}+\overline{h})}\,r_{J,h,C_i,C_j}(\ell,k\tc\si(C_i),\si(C_j))\,,
\end{align}
where $C_i,C_j\in[N]$, and $\ell,k\in[n]$,  denote the labels such that $\si_\ell(C_i)=\eta_i$ and $\si_k(C_j)=\eta_j$ in the correspondence \eqref{eq:corres}. Using this pointwise estimate in \eqref{DirDU} we obtain \eqref{eq:MLSI-MF11}. 
This ends the proof in the case $M\le 0$. The case $M\ge 0$ can be handled in the same way with the change of variable $\eta\mapsto -\eta$. Indeed, this global flip changes the measure from $\mu_{N,J,h,\r}$ to $\mu_{N,J,-h,1-\r}$. Thus we may apply the previous argument to the new variables.
This ends the proof of Theorem \ref{th:MLSI-MF}.
\end{proof}

\subsection{Proof of Theorem \ref{MLSI DOWN UP}}
Let \((W_t)_{t \geq 0}\) be a standard Brownian motion in $\bbR^L$,  adapted to a filtration \((\mathcal{F}_t)_{t \geq 0}\), and set $\tilde W_t = \L^{1/2}W_t$. The stochastic localization process for the canonical Ising measure $\nu_{L,M}$ is given by 
\begin{equation}\label{stocloc definition}
    \varrho_t(\eta) =\frac1{Z_t} \exp\left( -\frac{t - 1}{2} \langle \eta, \Lambda \eta \rangle + \langle y_t + w, \eta \rangle \right) \mathbf{1}(\eta \in \Omega_{L,M}),\quad t\ge 0\,,
\end{equation}
where $Z_t$ is the normalization so that $\varrho_t\in\cP(\O_{L,M})$, and  the random field \(y_t \in \mathbb{R}^L, t\ge 0\) evolves according to the stochastic differential equation
\begin{equation}\label{SDE yt}
    dy_t =d\tilde W_t +  \Lambda m(y_t)\,dt, \qquad y_0 = 0,
\end{equation}
with $m(y_t)$ given by the mean vector $m(y_t) := {\varrho_t}[\eta]$, that is for all $i\in[L]$,
\begin{equation}
    [m(y_t)]_i = {\varrho_t}[\eta_i] =\textstyle{ \sum_{\eta\in\O_{L,M}}}\eta_i\varrho_t(\eta).
\end{equation}
That the process is well defined can be checked by standard methods. Moreover,  $(\varrho_t, t\ge 0)$ is an $\{\cF_t\}$-martingale and therefore $\nu_{L,M} = \bbE[\varrho_t]$ for all $t\ge 0$, where $\bbE$ denotes the expectation with respect to the stochastic localization process. We refer to \cite{chen2022localization}; see also Lemma 3.1 in \cite{caputo2025factorizations}. 

At time \(t = 1\), the measure \eqref{stocloc definition} is a random product distribution restricted to \(\Omega_{L,M}\), 
\begin{equation}\label{independent product}
    \varrho_1(\eta) \propto \prod_{i=1}^L \exp\left( y_{1,i}\eta_i + w_i\eta_i \right) \, \mathbf{1}(\eta \in \Omega_{L,M}),
\end{equation}
where $y_{1,i}$ denotes the $i$-th component of the vector $y_1$.
That is, $\varrho_1$ is a Bernoulli product measure conditioned on having total magnetization $M$. Such measures are known to satisfy the Strong Rayleigh property; see \cite{borcea2009negative}. This implies strong log-concavity of its generating polynomial, which in turn yields a modified logarithmic Sobolev inequality (MLSI) for the associated Down-Up walk \cite{CGM}; see also \cite{EntInd22,bauerschmidt2024kawasaki}. We shall actually establish a stronger estimate,  namely entropy factorization. This allows us to use the arguments in \cite{caputo2025factorizations} which yield a simple proof of Theorem \ref{MLSI DOWN UP}. 

To introduce the factorization estimate, we formulate the Down-Up walk for an arbitrary measure $\mu\in\cP(\O_{L,M})$ as follows.
As above, we say that there is a ball at $j\in[L]$ if $\eta_j=+1$. Thus, every $\eta\in\O_{L,M}$ has $K:=(L+M)/2$ balls. According to the distribution $\mu$, balls are not distinguished. However, we artificially distinguish them by taking a uniformly random labeling of the $K$ balls, and write $X_k\in[L]$ for the position of the $k$-th ball. With slight abuse of notation we call again $\mu$ the induced distribution of the vector $X:=(X_1,\dots,X_K)$.
Let $\mu_i=\mu(\cdot\tc X_j,\,j\neq i)$ denote the distribution of $X_i$ conditionally on the position of all remaining balls. Now, it is not hard to see that the Down-Up walk for the original measure $\mu$, defined by the rates \eqref{RatesDU} with $\mu$ in place of $\mu_{\L,w}$, coincides with the Markov chain where with rate $1$ independently each variable $X_i$ is updated by replacing it with a sample from $\mu_i$, provided we observe the chain on functions that are symmetric under permutations of the ball labels. Thus, letting $\cS$ be the set of  symmetric functions $X\mapsto F(X)$, the generator of the chain becomes
 \begin{equation}\label{Ldumu}
    \cL^{\rm du} F(X) = \sum_{i=1}^K (\mu_i[F](X) - F(X)),
\end{equation}
for any $F\in\cS$, 
where $\mu_i[F](X)=\mu(F\tc X_j,\,j\neq i)$ is the conditional expectation of $F$ given all $X_j,\,j\neq i$. It follows 
that the Dirichlet form is given by 
 \begin{equation}\label{Ddumu}
    \cE^{\rm du}(F,G) = \sum_{i=1}^K \mu\left[\cov_{\mu_i}(F,G)\right],
\end{equation}
for any $F,G\in\cS$, where $\cov_{\mu_i}(F,G)=\mu_i[FG] - \mu_i[F]\mu_i[G]$ denotes the covariance with respect to $\mu_i$. 
Moreover, 
let $\ent_{\mu_i}(F)=\mu_i[F\log F]-\mu_i[F]\log\mu_i[F]$ denote the entropy functional with respect to $\mu_i$. Note that this is a function of the variables $X_j,\,j\neq i$ and that for any realization of these variables Jensen's inequality implies 
\begin{equation}\label{eq:jensenCov}
\mathrm{Ent}_{\mu_i}F\le \cov_{\mu_i}(F,\log F)\,,
\end{equation}
for all $F\ge 0$. In light of \eqref{Ddumu}-\eqref{eq:jensenCov} the following factorization property implies the MLSI for $\mu$ with constant $1$. We say that $\mu\in\cP(\O_{L,M})$ is strongly log-concave
if its generating polynomial is strongly log-concave; see \cite{CGM}. 
 \begin{proposition}\label{th:tenso}
 Suppose $\mu\in\cP(\O_{L,M})$ is strongly log-concave. Then, for any $F\in\cS$, 
   \begin{equation}\label{eq:tenso}
 \ent_\mu F\le \sum_{i=1}^K \mu\left[\ent_{\mu_i}F\right].
\end{equation}
 \end{proposition}
\begin{proof}
From \cite[Lemma 11]{CGM}, or equivalently \cite[Theorem 5]{EntInd22}, any strongly log-concave $\mu$ satisfies the subadditivity inequality, 
  \begin{equation}\label{eq:tenso1}
 \ent_\mu F\ge \sum_{i=1}^K \ent_\mu F_i,\qquad F_i = \mu[F\tc X_i]\,,
\end{equation}
for any nonnegative symmetric function $F$. Note that by symmetry $\ent_\mu F_i=\ent_\mu F_1$ for all $i$. Moreover, it is known that the inequality \eqref{eq:tenso1} continues to hold with $\mu$ replaced by the conditional distribution $\mu(\cdot\tc X_{1},\dots,X_{k})$ for any $k=1,\dots,K-1$ and any realization of the variables $X_j$, $j=1,\dots,k$. 
The factorization statement  \eqref{eq:tenso} then follows from a simple recursive argument, see e.g.\ \cite[Section 4.1]{EntLecNotes}
\end{proof}
Since the measure $\varrho_1$ in \eqref{independent product} is strongly log-concave \cite{borcea2009negative}, we may apply Theorem \ref{th:tenso} with $\mu=\varrho_1$. Taking the expectation with respect to the stochastic localization process, 
\begin{equation}\label{eq:tenso11}
\bbE \left[\ent_{\varrho_1}F\right] \le \sum_{i=1}^K \bbE\left[\varrho_1\left[\ent_{\varrho_{1,i}}F\right]\right],
\end{equation}
where $\varrho_{1,i}=\varrho_{1}(\cdot\tc X_j,\,j\neq i)$. 
We would like to turn the statement \eqref{eq:tenso11} into a factorization statement for our original measure $\nu:=\nu_{L,M}= \bbE[\varrho_1]$. 
The first step is a convexity estimate.
\begin{lemma}\label{lem:convex}
For any nonnegative symmetric function $F$, for all $i\in[K]$, one has
  \begin{equation}\label{eq:subma}
  \bbE\left[\varrho_1\left[\ent_{\varrho_{1,i}}F\right]\right]\le \nu\left[\ent_{\nu_i} F \right]\,.
\end{equation}
\end{lemma}
\begin{proof}
Using the martingale property $\nu=\bbE[\varrho_1]$, 
 \begin{equation}\label{eq:subma1}
\bbE\left[\varrho_1\left[\ent_{\varrho_{1,i}}F\right]\right] = \nu[F\log F] - \bbE\left[\varrho_1\left[\varrho_{1,i}F\log \varrho_{1,i}F\right]\right]\,.
\end{equation}
We write 
 \begin{equation}\label{eq:subma2}
 \varrho_{1,i}F (X) = \frac{\varrho_{1}(F;X_{\neq i}) }{\varrho_{1}(X_{\neq i})},
\end{equation}
where we use the notation $X_{\neq i}=(X_j,\,j\neq i)$, $\varrho_{1}(X_{\neq i}) = \sum_{X_i}\varrho_{1}(X)$, and  $\varrho_{1}(F;X_{\neq i}) = \sum_{X_i}F(X)\varrho_{1}(X)$, and therefore 
\[
\bbE\left[\varrho_1\left[\varrho_{1,i}F\log \varrho_{1,i}F\right]\right] = \sum_{X_{\neq i}}\bbE\left[\varrho_{1}(F;X_{\neq i}) \log \frac{\varrho_{1}(F;X_{\neq i}) }{\varrho_{1}(X_{\neq i}) }\right]\,.
\]
Using $\nu=\bbE[\varrho_1]$,  and the joint convexity of the function $(x,y)\mapsto x\log(x/y)$ on $(0,\infty)^2$, by Jensen's inequality 
\[
\bbE\left[\varrho_1\left[\varrho_{1,i}F\log \varrho_{1,i}F\right]\right]\ge \sum_{X_{\neq i}}\nu(F;X_{\neq i}) \log \frac{\nu(F;X_{\neq i}) }{\nu(X_{\neq i}) }= \nu\left[\nu_{i}F\log \nu_{i}F\right]\,,
\]
where we use the notation \eqref{eq:subma2} for $\nu$ in place of $\varrho_1$. The desired bound \eqref{eq:subma} follows from \eqref{eq:subma1}.  
\end{proof}
%

Next,  we use the so-called entropic stability estimate from \cite{chen2022localization}, which allows one to compare relative entropies with respect to $\varrho_1$ with relative entropies with respect to $\nu$ by using a covariance estimate. 
Let 
\begin{align}\label{eq:covtv}
\cov[\mu](i,j) = \mu[\eta_i\eta_j]- \mu[\eta_i]\mu[\eta_j],
\end{align}
denote the $L\times L$ covariance matrix of any  $\mu\in\cP(\O_{L,M})$, and  write $\l(\cov[\mu])$ for its largest eigenvalue. If $v \in \bbR^{L}$, the tilted measure $\mathcal{T}_v\mu\in\cP(\O_{L,M})$ is defined as 
\begin{align}\label{eq:tvnu}
\mathcal{T}_v\mu(\eta) \propto\mu(\eta)e^{\scalar{v}{\eta}}\,.
	\end{align}
The next lemma is a version of \cite[Proposition 39]{chen2022localization}; see also \cite[Lemma 3.2]{caputo2025factorizations}. The proof can be adapted to our conservative setting without modifications.
\begin{proposition}\label{prop:ent_stab}
	Suppose there exists a deterministic function $\kappa:[0,1]\to \bbR_+$ such that  $\l(\cov[\cT_v\rho_t]) \leq \kappa(t)$ a.s., for all $t \in [0,1],$ for all $v\in\bbR^{L}$.
	Then,  for all functions $F\ge 0$,
	\begin{align}
		\bbE\left[\ent_{\rho_1}F\right] \geq 
		\exp\left(-\l(\L)\int_0^1\kappa(s)ds\right)\,\ent_\nu F\,. 
	\end{align}	
\end{proposition}
We turn to a bound on the covariance.  The following was established  
 in \cite[Lemma 50]{chen2022localization} for the non conservative setting, based on a decomposition introduced in 
 \cite{BB19}; see also \cite[Lemma 3.4]{caputo2025factorizations}. We adapt the proof given in \cite{caputo2025factorizations} to the present conservative setting.

\begin{proposition}\label{prop_cov}
Let $\nu =\nu_{L,M}$ be the canonical Ising measure \eqref{RestrDU} with positive definite interaction matrix $\L$
and external fields $w\in\bbR^L$. If $\l(\L)<1/2$, then the largest eigenvalue of $\cov[\nu]$ satisfies
\begin{align}
	\l(\cov[\nu]) \leq  \frac{2}{1-2\l(\L)}\,.	\end{align}	
\end{proposition}
\begin{proof}
By gaussian integration, for any $\eta\in\O_{L,M}$, 
	\begin{align}\label{eq:covrov12}
		\exp(\tfrac12\scalar{\eta}{\L\eta})  \propto 
		 \int_{\bbR^{L}}dy\exp\left(-\tfrac12\scalar{y}{\L^{-1}y} +\scalar{\eta}{y} \right),
	\end{align}
where the implicit positive constant is independent of $\eta$. Define the function
\begin{align*}
		f(y) := \exp\left(-\tfrac12\scalar{y}{\L^{-1}y} + \log\left(Z\left(y + w\right)\right)\right)\,, 
	\end{align*}
where $Z\left(v\right):=\sum_{\eta \in \O_{L,M}}\exp\left(\scalar{\eta}{v}\right)$, for any $v\in\bbR^{L}$.
From \eqref{eq:covrov12} we obtain 
\begin{align*}
		\nu(\eta) &\propto \int_{\bbR^{L}}dyf(y)Z\left(y + w\right)^{-1}
\exp\left( \scalar{\eta}{y + w} \right){\bf 1}(\eta \in \O_{L,M})\,.
	\end{align*}
		If $\mu$ denotes the uniform probability measure over $\O_{L,M}$ then for all $v\in\bbR^L$, \[
		\mathcal{T}_{v}\mu(\eta)=Z\left( v\right)^{-1}\exp\left( \scalar{\eta}{v} \right){\bf 1}_{\{\eta \in \O_{L,M}\}}\,.\] 
		Therefore, $\nu$ is given by a mixture of tilted measures $\mathcal{T}_{v}\mu$, 
\begin{align}\label{eq:decrho}
\nu(\eta) = \int \mathcal{T}_{y + w}\mu(\eta)\hat f(y)dy\,,\end{align}
where $\hat f:=f/\int dyf(y)$ is a probability density on $\bbR^{L}$.
By the law of total covariance,
	\begin{align}\label{eq:covtot}
		\cov[\nu] &= \int   dy\hat f(y) \cov [\mathcal{T}_{y + w}\mu]
		 \,+\cC_\nu\,,
	\end{align}
	where 
	$\cC_\nu$ stands for the $L\times L$ matrix
	\[
	\cC_\nu(i,j):= \int  \g_i(y)\g_j(y) \hat f(y) dy - \left(\int \g_i(y) \hat f(y) dy\right)\left(\int \g_i(y)\g_j(y) \hat f(y) dy\right),
	\]
	and we use the notation $\g_i(y):=\sum_{\eta\in\O_{L,M}}\mathcal{T}_{y + w}\mu(\eta)\eta_i$, $i\in[L]$. 
Next, we observe that $\l(\cov [\mathcal{T}_{v}\mu])\le 2$ for any $v\in\bbR^L$. Indeed, $\mathcal{T}_v\mu$ is Strong Rayleigh \cite{borcea2009negative}, which implies negative correlation, and thus for any $u\in\bbR^L$, letting $C=\cov [\mathcal{T}_{v}\mu]$, one
has
\[
\sum_j |C_{i,j}| = C_{i,i} - \sum_{j:\,j\neq i} C_{i,j} = 2C_{i,i}\,,
\]
for any $i\in[L]$, where we use $C_{i,j}\le 0$ for $i\neq j$ and the fact that $\sum_{j} C_{i,j}=0$. Since $\l(C)\leq \max_i\sum_j |C_{i,j}|$ and $C_{i,i}\le 1$ for any $i\in[L]$ one has the desired bound $\l(\cov [\mathcal{T}_{v}\mu])\le 2$.

We turn to a bound on the matrix $\cC_\nu$. For any $u \in\bbR^{L}$ we write
	\begin{align}\label{eq:covy}
		\scalar{u}{\cC_\nu u} &= \textstyle{\sum_{i,j}}u_iu_j\cC_\nu(i,j)=\var\left(g_u(Y)\right)\,, 
			\end{align}
where $Y$ is the $\bbR^{L}$-valued random variable with density $\hat f$, 
and 
$g_u(y):=\sum_{i} u_i  \g_i(y)$. 
We are going to use the Brascamp-Lieb inequality (Theorem 4.1 in \cite{brasclieb}). 
Namely, suppose that 
$-{\rm Hess}\log \hat f (y)$ is positive definite, with  lower bound $a\,{\rm Id}$ where $a>0$ and ${\rm Id}$ is the $L\times L$ identity matrix. Then, 
	\begin{align}\label{eq:BL}
	\var\left(g_u(X)\right)\le \frac1a\int dy \hat f(y)\|\nabla g_u(y)\|_2^2 \,.
	\end{align}
Note that, 
\[
\nabla_i g_u(y) = \sum_{j} u_j  \frac{\partial\g_j(y)}{\partial y_i}= \sum_{j}u_i C_{i,j}\,,
\]
where $C=\cov [\mathcal{T}_{y+w}\mu]$.  As discussed above, one has $\l(C)\le 2$.
Therefore,  for all $y\in\bbR^{L}$,  
\[
	\|\nabla g_u(y)\|_2^2=\scalar{Cu}{Cu}\le 4\scalar{u}{u}\,.
	\] 
From \eqref{eq:BL} we conclude 	
\begin{align}\label{eq:BL2}
	\scalar{u}{\cC_\nu u} \le \frac{4}{a} \scalar{u}{u}\,\,.
	\end{align}	
	It remains to establish the bound $-{\rm Hess}\log \hat f (y)\ge a\,{\rm Id}$. 
	Since ${\rm Hess}\log\left(Z\left(y + v\right)\right) =  \cov[\mathcal{T}_{y + w}\mu]$,  one has
	 ${\rm Hess}\log\left(Z\left(y + w\right)\right)\le  2\,{\rm Id}$. Moreover, ${\rm Hess}\scalar{y}{\L^{-1}y}\ge 2\l(\L)^{-1}{\rm Id}$. Therefore we obtain the quadratic form bound
		\begin{align}\label{eq:BLineq3}
-{\rm Hess}\log \hat f(y) )= {\rm Hess}\left(\tfrac{1}2\scalar{y}{\L^{-1}y} - \log\left(Z\left(y + v\right)\right) \right)\ge 
\left(\frac{1}{\l(\L)} -2\right){\rm Id}.
	\end{align}
	Summarizing, we may take $a =  (1-2\l(\L))/\l(\L)$. Since $4/a + 2 =  2/(1-2\l(\L))$, the desired conclusion follows from  \eqref{eq:covtot}. 
\end{proof}

\medskip

We now have all ingredients to complete the proof of Theorem \ref{MLSI DOWN UP}. From Proposition \ref{prop_cov} applied with $\L$ replaced by $(1-t)\L$, $t\in[0,1]$, and using the expression \eqref{stocloc definition} for $\varrho_t$ we see that the assumption of  Proposition \ref{prop:ent_stab} holds with 
\[
\kappa(s) = \frac{2}{1 - 2(1 - s)\lambda(\Lambda)}, \quad \forall s \in [0,1].
\]
Integrating, $\int_0^1\kappa(s)ds = -\frac1{\l(\L)}\log(1-2\l(\L))$. Therefore, Proposition \ref{prop:ent_stab} implies 
\begin{equation}
    \mathrm{Ent}_{\nu}F \leq \frac{1}{1 - 2\lambda(\Lambda)} \, \mathbb{E} \left[ \mathrm{Ent}_{\varrho_1}F \right],
\end{equation}
for all $F : \Omega_{L,M} \to \mathbb{R}_+$.
This, together with \eqref{eq:tenso11} and \eqref{eq:subma}, proves that 
\begin{equation}
    \mathrm{Ent}_{\nu}F \leq \frac{1}{1 - 2\lambda(\Lambda)} \, \sum_{i=1}^K \nu\left[ \mathrm{Ent}_{\nu_i}F \right],
\end{equation}
Using \eqref{eq:jensenCov},
it follows from \eqref{Ddumu}
that 
\begin{equation}
    \mathrm{Ent}_{\nu}F \leq \frac{1}{1 - 2\lambda(\Lambda)} \,
    \cE^{\rm du} (F,\log F) 
\end{equation}
This ends the proof of Theorem \ref{MLSI DOWN UP}.

\subsection{Multicomponent mean field exchange dynamics}
We extend the previous analysis to the case where the system consists of several regions each of which is restricted to have a fixed magnetization. To this end, consider again the probability measure $\mu_{\L,w}$ on $\{-1,+1\}^L$ given in \eqref{L-Ising-Measure}.
Let $\cB$ denote a partition of $[L]$, that is $\cup_{B\in\cB}B = [L]$ and $B\cap B'=\emptyset$ for all $B\neq B'$. Let also $\cM=(M_B, \,B\in\cB)$, $M_B\in\{-|B|,\dots,|B|\}$,  denote the magnetization vector, and define
\[
    \Omega_{L,\cM} := \left\{ \eta \in \{-1,+1\}^{L} :\textstyle{ \sum_{i \in B}} \eta_i = M_B, \quad \forall B \in \mathcal{B} \right\}\,,
\]
and the multicomponent canonical measure
\[
\nu_{L,\cM}: =  \mu_{\L,w}\left(\cdot\tc\O_{L,\cM}\right)\,.
\]
For every $B\in\cB$ we define a Down-Up walk with rates
\begin{equation}\label{RatesDUB}
c^{\rm du}_B(\eta,\eta^{i,j}) = \mathbf{1}(i,j\in B)\mathbf{1}(\eta_i = +1,\eta_j = -1)
\frac{\mu_{\L,w}(\eta^{i,j})}{\sum_{k \in U_{B,i}(\eta)} \mu_{\L,w}(\eta^{ik})},
\end{equation}
where \(U_{B,i}(\eta) = \{j \in B : \eta_j = -1\} \cup \{i\}\). 
Consider now the Down-Up walk given by the generator \eqref{GenDU} with rates 
\[
c^{\rm du}(\eta,\eta^{i,j}):=\sum_{B\in\cB}c^{\rm du}_B(\eta,\eta^{i,j}).
\] 
It is not hard to check that, for any choice of the magnetization vector $\cM$, the process is irreducible in the state space $\O_{L,\cM}$, with unique reversible and invariant measure $\nu_{L,\cM}$. 
The Dirichlet form is given by 
\begin{gather}\label{DirDUB}
\mathcal{E}^{\rm du}_\cB(F,G) =  
\frac{1}{2}\sum_{B\in\cB}\sum_{i,j\in B} \nu_{L,\cM}
\left[ c^{\rm du}_B(\eta,\eta^{i,j})\nabla^{i,j}F(\eta)\nabla^{i,j}G(\eta) \right]\,, 
\end{gather}
for \(F,G : \Omega_{L,\cM} \to \mathbb{R}\).
Let also $\ent_{\nu_{L,\cM}}(F)$ denote the entropy with respect to the multi-canonical Ising measure $\nu_{L,\cM}$.
The multicomponent version of Theorem \ref{MLSI DOWN UP} is as follows.

\begin{theorem}[MLSI for Multicomponent Down-Up walk]\label{MLSI DOWN UP MC}
Let $\cB$ be a fixed partition of $[L]$. Suppose that \(\Lambda\) is a nonnegative definite $L\times L$ matrix with largest eigenvalue \(\lambda(\Lambda) < 1/2\). Then 
for all values of the vector $\cM$, for all $w\in\bbR^L$,
\begin{equation}\label{MLSIB}
\mathcal{E}^{\rm du}_\cB(F, \log F)\ge (1 - 2\lambda(\Lambda))^2\, \mathrm{Ent}_{\nu_{L,\cM}} (F) \,,
\end{equation}
for every \(F : \Omega_{L,\cM} \to \bbR_+\).
\end{theorem}
The key to prove the above theorem is the following factorization statement. For any $B\in\cB$ let $\nu_B=\nu_{L,\cM}(\cdot\tc \eta_j,j\notin B)$ denote distribution inside $B$ conditioned on the outside variables, and let  
$\mathrm{Ent}_{\nu_B}(F)=\nu_B(F\log F)-\nu_B(F)\log\nu_B(F)$ be the associated entropy functional. Note that $\mathrm{Ent}_{\nu_B}(F)$ is a function of the outside variables $\eta_j,j\notin B$.
\begin{proposition}[Approximate Factorization of Entropy]\label{approximate tensorization of entropy over components}
Under the assumptions of Theorem \ref{MLSI DOWN UP MC}, for all $F:\Omega_{L,\cM} \to \bbR_+$, 
\begin{equation}\label{MLSIBa}
\sum_{B \in \mathcal{B}} \nu_{L,\cM}\left[\operatorname{Ent}_{\nu_B}(F)\right]\ge (1 - 2\lambda(\Lambda))\, \mathrm{Ent}_{\nu_{L,\cM}}(F) \,.
\end{equation}
\end{proposition}
\begin{proof}
The proof can be obtained as in \cite[Theorem 2.8]{caputo2025factorizations}, which establishes a Shearer-type inequality for multicomponent systems. For the reader's convenience we sketch the main steps. Consider the measure $\varrho_t$ defined as in \eqref{stocloc definition}
with the set $\O_{L,M}$ replaced by $\O_{L,\cM}$. A key feature of the stochastic localization process in this case is that the measure $\varrho_1$ is now a product over $B\in\cB$, namely
\[
\varrho_1 = \bigotimes_{B\in\cB} \varrho_{1,B}\,,\qquad \varrho_{1,B} = \cT_{y_{1,B}+w_B}\mu_B\,,
\]
where $\mu_B$ is the uniform distribution over $\eta\in \{-1,+1\}^B$ such that $\sum_{j\in B}\eta_j=M_B$, and $y_{1,B},w_B$ denote the projection of the $L$-vectors $y_1,w$ onto $B$. This can be seen from \eqref{independent product} where now $\O_{L,M}$ is replaced by $\O_{L,\cM}$. 

By the tensorization property of product measures, see e.g.\ \cite{EntLecNotes}, one has the factorization
\begin{equation}\label{eq:tensoB}
\ent_{\varrho_1}F \le \sum_{B\in\cB} \varrho_1\left[\ent_{\varrho_{1,B}}F\right]\,.
\end{equation}
Moreover, we may apply the convexity argument in Lemma \ref{lem:convex} to obtain, for any $B\in\cB$, 
\begin{equation}\label{eq:tensoB1}
\bbE\left[ \varrho_1\left[\ent_{\varrho_{1,B}}F\right]\right]\le \nu_{L,\cM}\left[\ent_{\nu_B}F\right]\,.
\end{equation}
At this point we repeat the steps in  Proposition \ref{prop:ent_stab} and Proposition \ref{prop_cov}. The fact that 
the measures $\cT_{v}\mu_B$, $B\in\cB$, are Strong Rayleigh for any $v\in\bbR^B$ implies that their product is also  Strong Rayleigh, and thus the same covariance estimates can be used. In conclusion,
\[
\sum_{B \in \mathcal{B}} \nu\left[\operatorname{Ent}_{\nu_B}(F)\right] \,\ge\, \mathbb{E}_\mu\left[\operatorname{Ent}_{\varrho_{1}}(F)\right]\,\ge\,
(1 - 2\lambda(\Lambda))\, \mathrm{Ent}_{\nu_{L,\cM}}(F) \,.
\]
\end{proof}

\begin{proof}[Proof of Theorem \ref{MLSI DOWN UP MC}]
For every $B\in\cB$, the measure $\nu_B$ is of the form of a canonical Ising measure as in \eqref{RestrDU} with $L=|B|$, $\L$ replaced by $\L_B=(\L_{i,j},\,i,j\in B)$, the principal submatrix corresponding to the set $B$, and $w$ replaced by some new external field $w(\eta_{B^c})=(w_i(\eta_{B^c}),\,i\in B)$,  which depends on the outside variables $   \eta_{B^c}=(\eta_j, \,j\notin B)$ through
\[
w_i(\eta_{B^c})=w_i+\sum_{j\notin B}\L_{i,j}\eta_j\,.
\]
 Thus we may apply Theorem \ref{MLSI DOWN UP} to this system. Noting that the rates $c^{\rm du}_B(\eta,\eta^{i,j})$ can be written as in \eqref{RatesDUB} with $\mu_{\L,w}$ replaced by $\nu_B$, this gives, for all $B$, uniformly in $\eta_{B^c}$, 
\begin{gather}\label{DirDUB2}
\frac{1}{2}\sum_{i,j\in B} \nu_{B}
\left[ c^{\rm du}_B(\eta,\eta^{i,j})\nabla^{i,j}F(\eta)\nabla^{i,j}\log F(\eta) \right]\ge \,(1 - 2\lambda(\Lambda_B))\, \mathrm{Ent}_{\nu_B}(F)\,.
\end{gather}
Summing over $B\in\cB$, taking the expectation with respect to $\nu_{L,\cM}$, and noting that $\lambda(\Lambda_B)\le \lambda(\Lambda)$,
\begin{gather}\label{DirDUB3}
\mathcal{E}^{\rm du}_\cB(F, \log F)\ge (1 - 2\lambda(\Lambda))\sum_{B\in\cB}
\nu_{L,\cM}\left[\mathrm{Ent}_{\nu_B}(F)\right].
\end{gather}
The statement in Theorem \ref{MLSI DOWN UP MC} then follows from Proposition \ref{approximate tensorization of entropy over components}.
\end{proof}

\subsection{Proof of Theorem \ref{th:MLSI}}
We apply Theorem \ref{MLSI DOWN UP MC} with the notation given by the correspondence in \eqref{eq:corres}, and with the following choice of blocks $B$.
Given a partition \(\mathcal{A}\) of \([n]\), we define the associated  partition \(\mathcal{B} = (B(A),\,A\in\cA)\) of \([Nn]\) where, for each $A\in\cA$,
$B(A)$ is the stripe $[N]\times A$, or more formally
\[
    B(A) := \{i \in [Nn] : i = \ell + qn,\ \ell \in A,\ q \in \mathbb{Z}_+,\ 1\le\ell + qn \leq Nn\}.
\]
Note that $\{B(A), A\in\cA\}$ forms a partition of $[Nn]$, and if $B=B(A)$, then $\eta_B$ corresponds to the spins $(\si_\ell(i), \ell\in A,i\in[N])$. Let $\mu$ be the canonical Ising measure defined in \eqref{eq:pbm}, where $\r\in \cD_{\cA,N}$, and for $B=B(A)$, write $\mu_B$ for the measure 
\[
\mu_B = \mu(\cdot\tc \si_\ell(i), \ell\notin A, i\in[N]),
\] that is $\mu_B$ is the measure obtained by conditioning on all variables outside of $B$. 
The factorization in Proposition \ref{approximate tensorization of entropy over components} is equivalent to 
\begin{equation}\label{MLSIB2}
\sum_{B \in \mathcal{B}} \mu\left[\operatorname{Ent}_{\mu_B}(F)\right]\ge (1 - 2\lambda(\Lambda))\, \mathrm{Ent}_{\mu}(F) \,.
\end{equation}
Recalling that $\cK$ is now given by $\cK(\ell,k)=\frac1{|A|}$ if $\ell,k\in A$ and $\cK(\ell,k)=0$ otherwise, the Dirichlet form  
\eqref{Dirichlet form} becomes 
\begin{gather}\label{Dirichlet form2}
\mathcal{E}_{N, J,\r}(F, G) = \frac1n\sum_{A\in\cA}\mu\left[\cE_{B(A)}(F,G)\right]\,,\\
\cE_{B(A)}(F,G):=\frac{1}{2N|A|} \sum_{i,j=1}^N \sum_{\ell,k \in A} \mu_{B(A)}\left[ \hat p_J(\ell,k\tc \sigma(i), \sigma(j)) \nabla^{i,\ell;j,k}F\nabla^{i,\ell;j,k}G \right],
\end{gather}
for functions \(F, G : \O_{N,\r} \to \mathbb{R}\). 
We note that every $\mu_B$ is of the form \eqref{eq:pbm2}, with $(n,J,h)$ replaced by $(|A|,J_A,h_A)$, where $J_A=(J_{i,j}, \,i,j\in A)$, and $h_A\in\bbR^{N|A|}$ is the external field induced by the boundary condition outside of $B$, which is given by $[h_A(i)]_\ell= \sum_{k\notin A}J_{\ell,k}\si_k(i)$, $i\in [N]$, for every $\ell\in A$. 
Since $\bar h_A\le \bar J$, and $\l(\L)=\l(J)\ge \l(J_A)$, for all $A\subset [n]$, by Theorem \ref{th:MLSI-MF}, for all $A\in\cA$, 
\begin{equation}\label{eq:MLSI-MF11a}
\cE_{B(A)}(F, \log F) \ge \tfrac14(1 - 2\lambda(J))\,e^{-16\overline{J}}\operatorname{Ent}_{\mu_{B(A)}}(F),
\end{equation}
for all $F\ge 0$. 
From \eqref{MLSIB2} and \eqref{eq:MLSI-MF11} we conclude
\begin{gather}\label{Dirichlet form3}
\mathcal{E}_{N, J,\r}(F, \log F) \ge  \tfrac1{4n}(1 - 2\lambda(J))^2\,e^{-16\overline{J}}\operatorname{Ent}_{\mu}(F)\,.
\end{gather}
This ends the proof of Theorem \ref{th:MLSI}.

\section{Completing Kac's program}\label{sec:kac}
In this section we are going to prove Theorem \ref{Central Theorem}. The main step consists in establishing a suitable nonlinear modified log-Sobolev inequality. 
\subsection{Nonlinear MLSI}\label{sec:nmlsi}
Consider the quadratic system with kernel $\cQ^\cK_J$, and let $\cA$ be the corresponding partition of $[n]$. Let also $\cD_\mu$ 
denote the entropy dissipation functional from Lemma \ref{H-theorem}.  
\begin{definition}\label{def:nmlsi}
For all $h\in\G_\cA$, let $\cS_h$ be the set of $f:\O\to \bbR_+$ such that $\mu_{J,h}[f] = 1$ and such that $f\mu_{J,h}\in\cP(\O)$ has the same magnetization vector of $\mu_{J,h}$, that is \[m_\cA(\mu_{J,h})=m_\cA(f\mu_{J,h}).\]  The system is said to satisfy the {\em Nonlinear MLSI} with rate function \(\cI:\G_\cA\to \bbR_+\) if for all $h\in\G_\cA$, 
\begin{equation}\label{nonlinear log sobolev}
    \mathcal{D}_{\mu_{J,h}}(f)  \ge \cI(h)\,\mu_{J,h}\left[f\log f\right] 
\,,\qquad f\in\cS_h 
\,.
\end{equation}
\end{definition}

Recall the notation $\cP_\cA(\O)$ for the set of probability measures with nondegenerate magnetization, see Definition \ref{def:Amag}, and recall the uniqueness statement from Lemma \ref{lem:uni}.
\begin{lemma}\label{lem:nmlsi}
For any $p\in\cP_\cA(\O)$, let $h=h(m_\cA(p))$ denote the unique $h\in\G_\cA$ such that $m_\cA(\mu_{J,h})=m_\cA(p)$, and let \(p_t\) denote the solution of \eqref{non-linear rqs} with initial datum $p$. If the system satisfies the Nonlinear MLSI with rate function $\cI$, then 
\begin{equation}\label{eq:decex}
    H(p_t|\mu_{J,h}) \leq e^{-\cI(h) t} H(p|\mu_{J,h})\,,\qquad t\ge 0.
\end{equation}
\end{lemma}
\begin{proof}
Let $\mu=\mu_{J,h}$ and write $f_t=p_t/\mu$.  From Lemma \ref{conserved magn} it follows that $f_t\in\cS_h$ for all $t\ge 0$. 
By Lemma \ref{H-theorem}, and using the Nonlinear MLSI assumption, 
the function $\varphi(t):=H(p_t|\mu)=\mu[f_t\log f_t]$ satisfies 
\[
\frac{d}{dt}\,\varphi(t)  = - \cD_\mu(f_t)\le - \cI(h) \,\mu\left[f_t\log f_t\right]
=- \cI(h) \varphi(t)\,.
\] 
Integrating, we obtain \eqref{eq:decex}.
\end{proof}
To prove Theorem \ref{Central Theorem} it is then sufficient to prove the following {\em uniform} Nonlinear MLSI. 
\begin{theorem}\label{th:unmlsi}
In the setting of Theorem \ref{Central Theorem}, the system satisfies the Nonlinear MLSI with rate function $\cI$ such that
$\cI(h)\ge \a$, for all $h\in\G_\cA$, where $\a\ge \frac1{4n}(1-2\l(J))^2e^{-16\bar J}$.
\end{theorem}
The rest of this section is devoted to the proof of Theorem \ref{th:unmlsi}. The estimate will be derived using Theorem \ref{th:MLSI} together with Entropic Chaos and Fisher Chaos estimates that are at the heart of Kac's program.  We start with some preliminaries concerning the local central limit theorem.

\subsection{Local Central Limit Theorem and Kac Chaos}
Let \(\nu \in \mathcal{P}(\Omega)\) be a probability measure on \(\Omega = \{-1, +1\}^n\), and let \(\mathcal{A} \) be a partition of the set \([n]\). Define the \(|\mathcal{A}|\)-dimensional vector $M=(M_A, \,A\in\cA)$, where
\begin{align*}
M_A= 
\sum_{\ell \in A} \sigma_\ell, \qquad  A\in\cA\,.
\end{align*}
\begin{definition}\label{def:irred}
We say that $\nu\in \mathcal{P}(\Omega)$ is {\em irreducible} if for all $A\in\cA$, there exists an \(|\mathcal{A}|\)-dimensional vector $S$ such that 
$\nu(M = S)>0$ and $\nu(M = S+2\d_A)>0$, where 
\(\delta_A\) denotes the standard unit vector in the \(A\)-th coordinate. 
\end{definition}
Let \(V_\nu^\mathcal{A}\) denote the \textit{covariance matrix} of the vector \(M\), 
\begin{equation}
    V_\nu^\mathcal{A}(A, A') = \nu(M_A M_{A'}) - \nu(M_A)\nu(M_{A'}), \qquad A, A' \in \mathcal{A}.
\end{equation}
If $\nu$ is irreducible then 
\begin{align*}
\det V_\nu^\mathcal{A} > 0.
\end{align*}
Indeed, suppose that \(\det V_\nu^\mathcal{A} = 0\). Then there exists \(\overline{A} \in \mathcal{A}\) such that \(M_{\overline{A}}\) is almost surely a non-trivial linear combination of \(\{M_A : A\neq \bar{A}\}\) under \(\nu\). Consequently, no \(S \in \mathbb{R}^{\mathcal{A}}\) can exist such that both \(S\) and \(S + 2\delta_{A}\) have positive probability under \(\nu\), yielding a contradiction.
The following is a version of the local central limit theorem; see \cite[Proposition 2.2]{CP2024nonlinear} for a proof.

\begin{proposition}\label{CLT adapted}
Let \(\nu \in \mathcal{P}(\Omega)\) be irreducible, and let $M^{(N)}$ be the sum of $N$ independent copies of $M$, each distributed according to $\nu$. 
Then, there exists a constant \(K(\nu) > 0\) such that for all $N\ge 1$, 
\begin{align*}
\max_{S}\, \left|\, \bbP\big(M^{(N)} = S\big) - \frac{e^{-\frac{1}{2}\langle z_N, z_N \rangle}}{(2\pi N)^{|\mathcal{A}|/2} \sqrt{\det V^\mathcal{A}_\nu}} \,\right| \leq \frac{K(\nu)}{(2\pi N)^{|\mathcal{A}|/2 + 1}},
\end{align*}
where $z_N$ is the \(|\mathcal{A}|\)-dimensional vector defined as
\begin{equation}\label{z_N expression}
z_N = \frac{(V_\nu^{\mathcal{A}})^{-1/2}}{\sqrt{N}} \left(S - \mathbb{E}[M^{(N)}] \right),
\end{equation}
and the maximum is taken over all  \(|\mathcal{A}|\)-tuples \(S=(S_A, A\in\cA)\) such that for each $A\in\cA$, $S_A=2X_A-N|A|$ for some integer
$X_A\in\{0,\dots,N|A|\}$.
\end{proposition}
We apply the above statement with $n$ and $\cA$ fixed, while we take $N$  arbitrarily large. 
In particular, we use Proposition \ref{CLT adapted} to obtain a lower bound on the probability of the event $\O_{N,\r}$ defined in \eqref{eq:AmagO}, for suitable sequences $\r\in\cD_{\cA,N}$, when the spin configuration $(\si(i), \,i\in[N])$ is distributed according to the product measure $\nu^{\otimes N}$.

\begin{corollary}\label{cor order OmegamN}
Let \(\nu \in \mathcal{P}(\Omega)\) be irreducible and assume that $\r\in\cD_{\cA,N}$ satisfies the condition
\begin{align}\label{eq:admis}
\langle \r - \r_\mathcal{A}(\nu), \r - \r_\mathcal{A}(\nu) \rangle = O\left( 1/N \right),
\end{align}
where $\r_\mathcal{A}(\nu)=(1+m_\cA(\nu))/2$ is the density profile associated to $\nu$. 
Then,
\begin{align*}
\nu^{\otimes N} \left( \Omega_{N,\r}\right) \geq \frac{c}{N^{|\mathcal{A}|/2}},
\end{align*}
where $c=c(\nu)>0$ is a constant independent of $N$. In particular, 
\begin{align*}
\lim_{N \to \infty} \frac{1}{N} \log \left( \nu^{\otimes N} \left( \Omega_{N,\r}\right) \right) = 0.
\end{align*}
\end{corollary}
\begin{proof}
The event $\Omega_{N,\r}$ can be written as $M^{(N)}=S$ with $S_A=(2\r(A)-1) N|A|$, $A\in\cA$. Moreover, for each $A$, the $A$-component of $ \mathbb{E}[M^{(N)}]$ is given by $N|A|m_\cA(\nu)$. The 
results follow from Proposition \ref{CLT adapted} by noting that \eqref{eq:admis} is equivalent to  \(\langle z_N, z_N \rangle = O(1)\). 
\end{proof}

Given $\nu \in \mathcal{P}(\Omega)$, it is always possible to find $\r\in\cD_{\cA,N}$ such that \eqref{eq:admis} is satisfied. In fact, choosing 
\begin{align}\label{eq:canond}
\r(A) = \frac1{N|A|}\lfloor N |A|(1+m(A,\nu))/2 \rfloor\,,\qquad A\in\cA
\end{align}
one has  $\r(A) - \tfrac12(1+m(\nu,A)) = O(1/N)$, for all $A\in\cA$, 
and therefore \eqref{eq:admis} actually holds with $O(1/N^2)$ in place of $O(1/N)$. 

\begin{definition}[Canonical tensor product] \label{rem:density}
The choice \eqref{eq:canond} of $\r\in\cD_{\cA,N}$ will be called the {\em canonical density} associated to $\nu$. We also define the {\em canonical tensor product} associated to $\nu\in \mathcal{P}(\Omega)$ as the product measure 
$\nu^{\otimes N}$ conditioned on $\O_{N,\r}$, where $\r$ is the canonical density for $\nu$, 
\begin{align}\label{eq:canon}
\gamma_N(\nu)  =  \nu^{\otimes N}\left(
\cdot \tc \O_{N,\r}\right)\,.
\end{align}
\end{definition}
Kac's Chaos refers to the fact that marginals of the canonical tensor product converge as $N\to\infty$ to independent copies of $\nu$.  The next proposition quantifies this statement. For all fixed $k\in\bbN$, we write $P_k\g$ for the marginal of a measure $\g\in\cP(\O^N)$ on the first $k$ coordinates $\O^k$, for $k\le N$. We refer to \cite[Theorem 2.6]{CP2024nonlinear} for a proof. 
\begin{proposition}[Kac Chaos]\label{Kac Chaos}
Let \(\nu \in \mathcal{P}(\Omega)\) be irreducible. 
Then, for all 
$N\ge k\ge 1$, 
\begin{align*}
\|P_k \gamma_N(\nu) - \nu^{\otimes k}\|_{\mathrm{TV}} \leq \frac{Ck}{N},
\end{align*}
for some constant \( C = C(\nu) \). 
\end{proposition}

\subsection{Entropic  and Fisher Chaos} \label{sec:chaos}
We turn to the implications of the above estimates for the relative entropy and the Dirichlet form of the exchange dynamics. 
\begin{proposition}[Entropic Chaos]\label{Entropic Chaos}
Suppose $\nu_1,\nu_2 \in\mathcal{P}(\Omega)$ are irreducible measures such that $\nu_1\ll\nu_2$ and $m_\mathcal{A}(\nu_1)=m_\mathcal{A}(\nu_2)$. 
Then,
\begin{align*}
\lim_{N\to +\infty} \frac{1}{N}\, H_N\left(\gamma_N (\nu_1)\tc\gamma_N (\nu_2)\right) = H(\nu_1\tc\nu_2).
\end{align*}
\begin{proof}
By  definition of the canonical tensor products, calling $\r$ the common canonical density for $\nu_1,\nu_2$,  \[
H_N\left(\gamma_N (\nu_1)\tc\gamma_N (\nu_2)\right) 
=\gamma_N (\nu_1)\left[ \log\left(\frac{\nu_1^{\otimes^N}}{\nu_2^{\otimes^N}}\right)\right]\;-\;
\log\left(\frac{\nu_1^{\otimes^N}(\Omega_{N,\r})}{\nu_2^{\otimes^N}(\Omega_{N,\r}) }\right).
\]    
Corollary \ref{cor order OmegamN} implies \begin{equation}\label{ratio corollary}
\lim_{N\to +\infty} \frac{1}{N}\, \log\left(\frac{\nu_1^{\otimes^N}(\Omega_{N,\r})}{\nu_2^{\otimes^N}(\Omega_{N,\r}) }\right)
=0.
\end{equation}
Moreover, by symmetry 
\begin{align}
\frac{1}{N}\,\gamma_N (\nu_1)\left[ \log\left(\frac{\nu_1^{\otimes^N}}{\nu_2^{\otimes^N}}\right)\right]
&= P_1 \gamma_N (\nu_1)\left[ \log\left( \frac{\nu_1}{\nu_2} \right) \right],
\end{align}
and the conclusion follows by Proposition \ref{Kac Chaos}.
\end{proof}
\end{proposition}

Fisher Chaos is the convergence of the Dirichlet form \eqref{Dirichlet form} when the function $F$ is the density of a canonical tensor product. Recalling the notation from Definition \ref{def:nmlsi}, this can be formulated as follows. 
\begin{proposition}[Fisher Chaos]\label{Fisher Chaos}
    Fix $h\in\G_\cA$, and let $\mu:=\mu_{J,h}$. Suppose  $f\in\cS_h$ is such that the  measure \(\nu := f \mu \in \mathcal{P}(\Omega)\) is irreducible. 
        Let $\r$ denote the common canonical density of $\nu$ and $\mu$, and define
    \begin{equation}\label{observable approximat}
        F_N := \frac{\gamma_N(\nu)}{\gamma_N(\mu)}.
    \end{equation}
    Then,
    \begin{equation}\label{observable approximat2}
    \lim_{N \to \infty} \frac{1}{N}\, \mathcal{E}_{N, J,\r}(F_N, \log F_N) = 2\mathcal{D}_\mu(f).
    \end{equation}
\end{proposition}
\begin{proof}
For any $F:\O^N\to \bbR_+$ we write 
$
\cE_{N,J,\r}(F,\log F)=\frac1n\sum_{\ell,k\in[n]}\cK(\ell,k)\cE^{\ell,k}_{N,J,\r}(F),
$ where
\begin{eqnarray}
    \cE^{\ell,k}_{N,J,\r}(F):=-\frac{1}{N}\sum_{i,j=1}^N\mu_{N,J,\r}\left[\hat p_J(\ell,k|\sigma(i),\sigma(j))\nabla^{i,\ell;j,k}F\log F\right].
\end{eqnarray} 
Similarly, 
$\cD_\mu(f)=\frac1n\sum_{\ell,k\in[n]}\cK(\ell,k)\cD_\mu^{\ell,k}(f)$, where 
\begin{align}
    \cD_\m^{\ell,k}(f)&:=-\frac12\sum_{\tau,\tau'\in\Omega}\mu(\tau)\mu(\tau')p_J(\ell,k|\tau,\tau') \,\times\;\nonumber\\
    &\qquad \;\times \,\left(f\big(S_\ell(\tau,\tau'_k)\big)f\big(S_k(\tau',\tau_\ell)\big)-f(\tau)f(\tau')\right)\log\big(f(\tau)f(\tau')\big).
\end{align}
Therefore, 
to prove \eqref{observable approximat2} it is sufficient to show that for all fixed $n$ and $\ell,k\in[n]$, with $F_N$ as in \eqref{observable approximat},
    \begin{equation}\label{observable approximat21}
    \lim_{N \to \infty} \frac{1}{N}\,  \cE^{\ell,k}_{N,J,\r}(F_N) = 2\cD_\m^{\ell,k}(f).
    \end{equation}
 Since \(F_N\) and $\mu_{N,J,\r}$ are invariant under particle permutations, 
 \begin{align*}
  \frac1{N} \cE^{\ell,k}_{N,J,\r}(F_N)    = \frac{N - 1}{N} \cE_{1,2}(F_N) +  \frac1{N} \cE_{1}(F_N)\,,
\end{align*}
where
 \begin{gather*}
 \cE_{1,2}(F_N)  = -\mu_{N,J,\r}\left[p_J(\ell,k|\sigma(1),\sigma(2))\nabla^{1,\ell;2,k}F_N\log F_N\right]\,,
\\
 \cE_{1}(F_N)  = -\mu_{N,J,\r}\left[\hat p_J(\ell,k|\sigma(1),\sigma(1))\nabla^{1,\ell;1,k}F_N\log F_N\right]\,.
\end{gather*}
Thus, it suffices to show that 
\begin{gather}
    \lim_{N \to \infty}  \cE_{1,2}(F_N)= 2\cD_\m^{\ell,k}(f)\,,\label{eq:topo1}\\
    \lim_{N \to \infty}  \frac1N \,\cE_{1}(F_N)=0\label{eq:topo2}\,.
\end{gather}
For any configuration \(\sigma \in \Omega_{N,\r}\),
\[
F_N(\sigma) = \frac{\mu^{\otimes N}\big( \Omega_{N,\r} \big)}{\nu^{\otimes N}\big(\Omega_{N,\r} \big)} \prod_{s = 1}^N f\big( \sigma(s) \big),
\]
so that
\begin{align*}
\cE_{1,2}(F_N) 
&= -\sum_{\sigma \in \Omega_{N,\r}} \frac{\mu(\si(1))\mu(\si(2))\prod_{s \neq 1,2} \nu\big( \sigma(s) \big)}{\nu^{\otimes N}\big( \Omega_{\r_N}^{\mathcal{A}} \big)}\,
p_J(\ell,k\tc\si(1),\si(2))\,\log\left(F_N(\si) \right) \;\times
\\
&\qquad \times 
\left[ f( S_\ell( \sigma(1), \sigma(2)_k)) f( S_k( \sigma(2), \sigma(1)_\ell) ) - f( \sigma(1) ) f( \sigma(2) ) \right].
\end{align*}
We decompose $ \Omega_{N,\r}$ by fixing the values \(\t=\sigma(1), \t'=\sigma(2) \),
\begin{equation}\label{decompOmN}
 \Omega_{N,\r} = \cup_{\t,\t'\in \Omega} \Omega_{N,\r}( \t,\t'),
\end{equation}
where \(\Omega_{N,\r}( \t,\t')\) denotes the set of configurations \((\sigma(s))_{s \neq 1,2}\in\O^{N-2}\) such that
\[
\sum_{\ell \in A} \sum_{s \neq 1,2} \sigma_\ell(s)= 
(N - 2) |A| \left(2\widetilde{\r}(A\tc\t,\t')-1\right), \qquad A \in \mathcal{A},
\]
where the adjusted density vector \(\widetilde{\r}(\cdot\tc\t,\t')\in\cD_{\cA,N-2}\) satisfies 
\[
\left(2\widetilde{\r}(A\tc\t,\t')-1\right)= \frac{N}{N - 2}(2 \r(A)-1) -   \frac{1}{(N - 2) |A|} \sum_{\ell \in A} (\t_\ell + \t'_\ell)
.
\]
Note that $\widetilde{\r}(A\tc\t,\t') = \r(A) + O(1/N)$ for all $A$ and $\t,\t'\in\O$. In particular, by Proposition \ref{CLT adapted}, for all $\t,\t'\in\O$,
\begin{equation}\label{decompOmN1}
\lim_{N \to \infty} \frac{\nu^{\otimes {N-2}}( \Omega_{N,\r}( \t,\t'))}{\nu^{\otimes N}( \Omega_{N,\r})} = 1.
\end{equation}
From the decomposition \eqref{decompOmN}, 
\begin{align*}
&\cE_{1,2}(F_N) 
=-\sum_{\t,\t'\in\O}
 \mu(\t)\mu(\t')p_J(\ell,k\tc\t,\t')\left[ f( S_\ell(\t, \t'_k)) f( S_k( \t', \t_\ell) ) - f( \t ) f( \t' ) \right]
 \;\times
\\
&\; \times 
\sum_{\eta \in \Omega_{N,\r}( \t,\t')}\frac{\prod_{s \neq 1,2} \nu\big( \eta(s) \big)}{\nu^{\otimes N}\big( \Omega_{N,\r} \big)}
\,\left[\log \frac{\mu^{\otimes N}\big( \Omega_{N,\r} \big)}{\nu^{\otimes N}\big(\Omega_{N,\r} \big)}  + \log (f(\t)f(\t')) + \log\left(\prod_{s\neq 1,2}f(\eta(s)) \right)\right].
\end{align*}
Using reversibility as in Lemma \ref{lem:db}, we see that the function
\[
\phi(\t,\t'):=\mu(\t)\mu(\t')p_J(\ell,k\tc\t,\t')\left[ f( S_\ell(\t, \t'_k)) f( S_k( \t', \t_\ell) ) - f( \t ) f( \t' ) \right]
\]
is antisymmetric in the exchange $(\t,\t')\leftrightarrow(S_\ell(\t, \t'_k),S_k( \t', \t_\ell) )$. On the other hand, \[ 
 \Omega_{N,\r}( \t,\t')= \Omega_{N,\r}(S_\ell(\t, \t'_k),S_k( \t', \t_\ell) )\,.\]  It follows that
\begin{align*}
&\cE_{1,2}(F_N) 
=-\sum_{\t,\t'\in\O}\frac{\nu^{\otimes {N-2}}( \Omega_{N,\r}( \t,\t'))}{\nu^{\otimes N}( \Omega_{N,\r})} \,\phi(\t,\t') \log (f(\t)f(\t'))
.
\end{align*}
From \eqref{decompOmN1} we obtain
\begin{align*}
&\lim_{N\to\infty}\cE_{1,2}(F_N) 
=-\sum_{\t,\t'\in\O}\phi(\t,\t') \log (f(\t)f(\t')) = 2\cD_\mu(f)\,.
\end{align*}
This proves \eqref{eq:topo1}. To prove \eqref{eq:topo2}
we proceed in a similar way. We decompose $ \Omega_{N,\r}=\cup_{\t\in \Omega} \Omega_{N,\r}( \t)$ by fixing \(\t=\sigma(1) \),
where \(\Omega_{N,\r}( \t)=\Omega_{N-1,\widetilde\r(\cdot\tc\t)}\) denotes the set of configurations \((\sigma(s))_{s \neq 1}\in\O^{N-1}\) with fixed density $\widetilde\r(\cdot\tc\t)$ such that
$\widetilde{\r}(A\tc\t) = \r(A) + O(1/N)$ for all $A$ and $\t\in\O$.
Reasoning as above, 
it follows that
\begin{align*}
&\cE_{1}(F_N) 
=-\sum_{\t\in\O}\frac{\nu^{\otimes {N-1}}( \Omega_{N,\r}( \t))}{\nu^{\otimes N}( \Omega_{N,\r})} \,\phi(\t) \log (f(\t))
,
\end{align*}
where 
\[
\phi(\t)=\mu(\t)\hat p_J(\ell,k\tc\t,\t)\left[ f(\t^{\ell,k}) - f( \t ) \right]\,.
\]
Therefore, \[
\lim_{N\to\infty}
\cE_{1}(F_N)  = -\sum_{\t\in\O} 
\phi(\t) \log (f(\t))\,.
\]
This implies \eqref{eq:topo2}.
\end{proof}

\subsection{Proof of Theorem \ref{Central Theorem}}
As we have seen in Section \ref{sec:nmlsi}, it is sufficient to prove Theorem \ref{th:unmlsi}.
Take $f\in \cS_h$, for a fixed $h\in\G_\cA$, and write $F_N$ for the function in Proposition \ref{Fisher Chaos}. Set
$\mu=\mu_{J,h}$, and assume that $\nu=f\mu$ is irreducible in the sense of Definition \ref{def:irred}. Therefore, we know that
\begin{gather}\label{Dirichlet form31}
\cD_\mu(f) = \lim_{N\to\infty} \frac1N\,\mathcal{E}_{N, J,\r}(F_N, \log F_N)\,. 
\end{gather}
On the other hand, by Theorem \ref{th:MLSI}, $\mathcal{E}_{N, J,\r}(F_N, \log F_N) \ge \a\,\ent_{N, J,\r}F_N$, for all $N$, with $\a\ge \frac1{4n}(1-2\l(J))^2e^{-16\bar J}$.
By Proposition \ref{Entropic Chaos} with $\nu_1=\nu$ and $\nu_2=\mu$, one has the convergence \[
 \lim_{N\to\infty}\frac1N\,\ent_{N, J,\r}F_N 
= \mu[f\log f]\,.\] It follows that 
\begin{gather}\label{Dirichlet form32}
\cD_\mu(f) \ge \a\,\mu[f\log f] \,. 
\end{gather}
This establishes the desired bound for all $f\in\cS_h$ such that $\nu=f\mu$ is irreducible. If $\nu$ is not irreducible, we proceed by regularising it  as follows. 
For any \( \varepsilon \in(0,1) \), define
\[
f_\varepsilon = (1 - \varepsilon)f + \varepsilon.
\]
Note that \(f_\e\in\cS_h\) whenever $f\in\cS_h$. Moreover $\nu_\e:=f_\e\mu$ is irreducible since $\nu_\e(\t)\ge \e\mu(\t)>0$ for all $\t\in\O$.  
Thus, applying the above reasoning to $f_\e$ we obtain
\begin{gather}\label{Dirichlet form33}
\cD_\mu(f_\e) \ge \a\,\mu[f_\e\log f_\e] \,. 
\end{gather}
Taking the limit as $\e\to 0$, by continuity we obtain the desired inequality for $f$.
This proves $\cI(h)\ge \a$ for all $h\in\G_\cA$, which ends the proof.

\subsection{Open problem: exponential ergodicity at all temperatures?}

We conclude with the discussion of some open questions.  
The exponential decay established in Theorem~\ref{Central Theorem} holds under the assumption that the interaction matrix~$J$ is nonnegative definite and has spectral radius \( \lambda(J) < 1/2 \).  
The positivity assumption is, in principle, not restrictive, since one may freely modify the diagonal entries of~\( J \) without affecting the corresponding Ising measure.  
The spectral radius condition, however, imposes that the interaction 
be sufficiently weak, which we interpret as a high-temperature assumption.  

For instance,  in the case of the Curie-Weiss model, where 
 $J_{i,j} = \b/n$ for all $i,j\in[n]$, the condition $\l(J)<1/2$ imposes $\b<1/2$ which is a factor $2$ off with respect to the critical value $\b_c=1$. It would be certainly desirable to have exponential decay with rate of order $1/n$ throughout the whole region $\b\in(0,\b_c)$ in this case. 
On the other hand, understanding the behavior of the system in the low temperature region $\b>\b_c$ remains an open problem. 
In particular, it is not even clear whether one should expect a slowdown or not for the nonlinear equation.

 Nevertheless, by analogy with linear evolutions such as low-temperature Glauber or Kawasaki dynamics \cite{Martinelli}, one may ask whether exponential decay to equilibrium still holds for every interaction~$J$, possibly with rates that deteriorate exponentially in~$n$.  
Establishing such a result would amount to an exponential ergodicity property for the nonlinear equation, providing a quantitative counterpart to the plain convergence in Theorem \ref{th:basic}.  


A possible strategy toward this goal is to revisit Kac's program and show that a statement analogous to Theorem~\ref{th:MLSI} remains valid for arbitrary interactions~$J$, with some rate \( \alpha \ge c(J) \) for a constant \( c(J) > 0 \) that may depend on~$J$, possibly exponentially small in~$n$.  
By repeating the reasoning in Section~\ref{sec:MLSI}, and in particular, invoking Proposition~\ref{prop:ent_stab}, such a result would follow if one could prove that the tilted canonical Ising measure \( \mathcal{T}_v \mu_{N,\rho} \) satisfies a covariance bound of the form
\[
\lambda\!\left( \mathrm{Cov}[\mathcal{T}_v \mu_{N,\rho}] \right) \le C(J),
\]
for some constant \( C(J) \) that may grow with~$n$ but remains uniformly bounded in both the linear tilt~$v$ and the number of particles~$N$.  
We conjecture  that such a bound indeed holds for all interactions $J$. 
Using equivalence-of-ensembles estimates, as in Proposition~\ref{Kac Chaos}, one can obtain such bounds for each fixed choice of the tilt~$v$, uniformly in $N$.  
However, establishing estimates that are also uniform in the tilt parameter appears to be a rather challenging problem.

\bigskip

\subsection*{Acknowledgments}
We thank Lorenzo Bertini for valuable suggestions regarding the models studied in this paper. We are also grateful to Alistair Sinclair and Zongchen Chen for many stimulating and helpful discussions.

\end{document}